\documentclass[oneside,english]{amsart}
\usepackage[T1]{fontenc}
\usepackage[latin9]{inputenc}
\usepackage{babel}
\usepackage{array}
\usepackage{float}
\usepackage{units}
\usepackage{multirow}
\usepackage{amstext}
\usepackage{amsthm}
\usepackage{amssymb}
\usepackage{graphicx}
\usepackage[all]{xy}
\usepackage[unicode=true,pdfusetitle,
 bookmarks=true,bookmarksnumbered=false,bookmarksopen=false,
 breaklinks=false,pdfborder={0 0 1},backref=false,colorlinks=false]
 {hyperref}

\makeatletter

\providecommand{\tabularnewline}{\\}

\numberwithin{equation}{section}
\numberwithin{figure}{section}
\theoremstyle{plain}
\newtheorem*{thm*}{\protect\theoremname}
\theoremstyle{plain}
\newtheorem{thm}{\protect\theoremname}
\theoremstyle{plain}
\newtheorem{lem}[thm]{\protect\lemmaname}
\theoremstyle{plain}
\newtheorem{prop}[thm]{\protect\propositionname}
\theoremstyle{plain}
\newtheorem{cor}[thm]{\protect\corollaryname}
\theoremstyle{definition}
\newtheorem{defn}[thm]{\protect\definitionname}
\theoremstyle{remark}
\newtheorem{rem}[thm]{\protect\remarkname}
\theoremstyle{remark}
\newtheorem*{rem*}{\protect\remarkname}
\theoremstyle{definition}
\newtheorem{example}[thm]{\protect\examplename}
\theoremstyle{definition}
\newtheorem{problem}[thm]{\protect\problemname}

\makeatother

\providecommand{\corollaryname}{Corollary}
\providecommand{\definitionname}{Definition}
\providecommand{\examplename}{Example}
\providecommand{\lemmaname}{Lemma}
\providecommand{\problemname}{Problem}
\providecommand{\propositionname}{Proposition}
\providecommand{\remarkname}{Remark}
\providecommand{\theoremname}{Theorem}

\begin{document}

\title[countability conditions]{Comparison of countability conditions within three fundamental classifications
of convergences}

\author{Frédéric Mynard}
\begin{abstract}
The interrelations between various classes of convergence spaces defined
by countability conditions are studied. Remarkably, they all find
characterizations in the usual space of ultrafilters in terms of classical
topological properties. This is exploited to produce relevant examples
in the realm of convergence spaces from known topological examples.
\end{abstract}

\maketitle
\global\long\def\G{\mathcal{G}}
 \global\long\def\F{\mathcal{F}}
 \global\long\def\H{\mathcal{H}}
 \global\long\def\L{\operatorname{L}}
\global\long\def\U{\mathcal{U}}
\global\long\def\W{\mathcal{W}}
 \global\long\def\P{\mathcal{P}}
 \global\long\def\B{\mathcal{B}}
 \global\long\def\A{\mathcal{A}}
\global\long\def\D{\mathcal{D}}
\global\long\def\O{\mathcal{O}}
 \global\long\def\N{\mathcal{N}}
 \global\long\def\X{\mathcal{X}}
 \global\long\def\lm{\mathrm{lim}}
 \global\long\def\then{\Longrightarrow}

\global\long\def\V{\mathcal{V}}
\global\long\def\C{\mathcal{C}}
\global\long\def\adh{\mathrm{\mathrm{adh}}}
\global\long\def\Seq{\mathrm{Seq\,}}
\global\long\def\intr{\mathrm{int}}
\global\long\def\cl{\mathrm{cl}}
\global\long\def\inh{\mathrm{inh}}
\global\long\def\diam{\mathrm{diam}}
\global\long\def\card{\mathrm{card\,}}
\global\long\def\S{\operatorname{S}}
\global\long\def\T{\operatorname{T}}
\global\long\def\I{\operatorname{I}}
\global\long\def\BaseD{\operatorname{B}_{\mathbb{D}}}
\global\long\def\AdhD{\operatorname{A}_{\mathbb{D}}}
\global\long\def\K{\operatorname{K}}

\global\long\def\fix{\mathrm{fix\,}}
\global\long\def\Epi{\mathrm{Epi}}
\global\long\def\De{\operatorname{P}_{1}}

\section{Introduction}

Classically, topologies can be described in a variety of ways, including
by prescribing the convergence of filters. Continuity from this point
of view is simply preservation of limits. Relaxing the resulting axioms
of convergence yields the larger category of \emph{convergence spaces}
( \footnote{It should be emphasized that convergence spaces were not introduced
out of an unmotivated quest for generality, but rather because natural
phenomenon of convergence failed to be topological. In fact, they
were introduced by G. Choquet \cite{cho} because topologies turned
out to be inadequate in the study of so called upper and lower limits
of Kuratowski in spaces of closed subsets of a topological space. } ) and continuous maps, which behaves better from the categorical
viewpoint (\footnote{In particular, this category is Cartesian-closed, that is, has canonical
function space structures yielding an exponential law, and is extensional,
that is, has quotients that are hereditary. In other words, it is
a quasitopos.}) and thus allows for constructions and methods unavailable in the
realm of topologies. As a result, even if one is only interested in
the topological case, there is often much to gain in embedding a topological
problem in the larger context of convergence spaces. The method can
be compared to using complex numbers to solve problems formulated
in the reals. An extensive account of this approach can be found in
\cite{DM.book}.

In order to introduce convergence structures, let us start with convergence
of filters in a topological space, which is then given by 
\begin{equation}
x\in\lim\F\iff\F\supset\N(x),\label{eq:topconv}
\end{equation}
where $\N(x)$ is the neighborhood filter of $x$. Clearly, $\N(x)$
is then the smallest filter converging to $x$, and the map $\N(\cdot)$
associating with each $x\in X$ a filter additionally has a filter
base composed of open sets around $x$. 

More generally, such a map $\V(\cdot)$ associating with each $x\in X$
a filter $\V(x)$ (with $x\in\bigcap\V(x)$) defines a convergence
given by 
\begin{equation}
x\in\lim\F\iff\F\supset\V(x).\label{eq:pretopconv}
\end{equation}
The filter $\V(x)$, called \emph{vicinity filter, }is the smallest
filter convergent to $x$. A convergence defined this way is a \emph{pretopology.
}Note that while every topology defines a pretopology, there are non-topological
pretopologies (\footnote{For instance sequential adherence in a sequential non Fréchet-Urysohn
space defines a non-topological pretopology. 

While directed graphs cannot be represented as topological spaces
unless the relation is transitive, all directed graph can be seen
as a pretopology in which $\V(x)$ is the principal filter of $\{x\}\cup\{y:y\to x\}$.

An easy example on the plane is given by the Féron cross pretopology,
in which $\V((x_{0},y_{0}))$ is the filter generated by sets of the
form $\left\{ (x_{0},y):|y-y_{0}|<\epsilon\right\} \cup\left\{ (x,y_{0}):|x-x_{0}|<\epsilon\right\} $
for $\epsilon>0$.}). 

More generally, a convergence $\xi$ on a set $X$ is a relation between
points of $X$ and filters on $X$, denoted $x\in\lim_{\xi}\F$ (or
$x\in\lim\F$ if no confusion can occur) whenever $x$ and $\F$ are
related, subject to the following two axioms: for every point $x\in X$
and pair of filters $\F$ and $\G$ on $X$,
\begin{eqnarray*}
x & \in & \lim\{x\}^{\uparrow}\\
\F\subset\G & \then & \lim\F\subset\lim\G,
\end{eqnarray*}
where $\{x\}^{\uparrow}$ is the principal ultrafilter of $x$. Evidently,
pretopologies, \emph{a fortiori} topologies, satisfy these axioms.

Convergences have been classified according to various criteria. We
focus here on three directions of classification, all introduced by
S. Dolecki (See in particular \cite{dolecki1996convergence}, \cite{D.comp}
and \cite{D.covers}):
\begin{enumerate}
\item According to how many filters need to be given at each point to describe
the convergence. This leads to the notion of \emph{paving number }$\mathsf{p}$
introduced in \cite{D.covers}, and \emph{pseudopaving numbers} $\mathsf{pp}$
introduced in \cite{myn.completeness}. Pretopologies are exactly
the convergences of (pseudo) paving number 1. We will focus here on
convergences of countable (pseudo) paving number.
\item According to \emph{depth}, that is, the cardinality of sets of filters
for which $\lim$ and $\bigcap$ commute. Namely, in a pretopology,
\begin{equation}
\lim\bigcap_{\F\in\mathbb{D}}\F=\bigcap_{\F\in\mathbb{D}}\lim\F\label{eq:depthintro}
\end{equation}
 for \emph{any} family $\mathbb{D}$ of filters. A convergence is
$\kappa$-\emph{deep} if (\ref{eq:depthintro}) holds for every set
$\mathbb{D}$ of filters of cardinality at most $\kappa$. We focus
here on \emph{countably deep} convergences, and those \emph{countably
deep for ultrafilters}, that is, when $\mathbb{D}$ is restricted
to a countable set of ultrafilters. 
\item According to the class of filters whose adherences determine the convergence.
\emph{Adherence }of a filter $\F$ is given by $\adh\,\F=\bigcup_{\F\subset\H}\lim\H$.
In the case of a principal filter $\{A\}^{\uparrow}$, we write $\adh\,A$
for $\adh\{A\}^{\uparrow}$. It turns out \cite{dolecki1996convergence}
that the formula 
\begin{equation}
\lim\F=\bigcap\left\{ \adh\,\D:\D\in\mathbb{D},\D\#\F\right\} ,\label{eq:adhdeterminedintro}
\end{equation}
where $\D\#\F$ means that $D\cap F\neq\emptyset$ whenever $D\in\D$
and $F\in\F$, and $\mathbb{D}$ denotes a class of filters, characterizes
several important classes of convergences. Namely, if $\mathbb{D}$
is the class $\mathbb{F}_{0}$ of principal filters, convergences
satisfying (\ref{eq:adhdeterminedintro}) are exactly pretopologies.
If $\mathbb{D}$ is the class $\mathbb{F}$ of all filters, they are
exactly the \emph{pseudotopologies }introduced by G. Choquet \cite{cho}.
Considering the intermediate classes $\mathbb{F}_{1}$ of countably
based filters and $\mathbb{F}_{\wedge1}$ of filters closed under
countable intersections define the classes of \emph{paratopologies
}\cite{dolecki1996convergence} and\emph{ hypotopologies }\cite{D.comp}
respectively. Though there are other possibilities as well, we focus
here on the latter two classes. 
\end{enumerate}
The main goal of this paper is to study the interrelations between
the classes of convergences introduced along these three branches
of classification. Pretopologies are particular instances of each
notion, but they turn out to form exactly the intersection of several
pairs of the corresponding classes of convergences. For instance, 
\begin{thm*}[\emph{Corollary \ref{cor:cdeeppara}}]
 A convergence is a pretopology if and only if it is a countably
deep paratopology.
\end{thm*}
This applies to hypotopologies because:

\begin{thm*}[\textit{\emph{Theorem \ref{thm:hypocountablydeep}}}]
 Every hypotopology is countably deep.
\end{thm*}
On the other hand,
\begin{thm*}[\textit{\emph{Theorem \ref{thm:SS1}}}]
 Every countably pseudopaved pseudotopology is a paratopology.
\end{thm*}
Below is a diagram showing the main classes of convergences considered,
with pretopologies as their intersection and examples distinguishing
these classes.

\begin{figure}[H]
\begin{centering}
\includegraphics[scale=0.5]{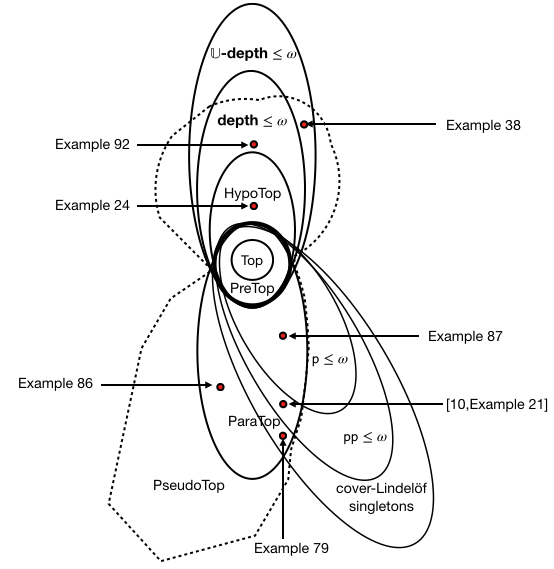}
\par\end{centering}
\caption{\label{fig:examples}Examples distinguishing the main classes}
\end{figure}

\medskip{}

It turns out that most of the classes considered are \emph{reflective},
that is, for each convergence $\xi$ there is the finest convergence
$R\xi$ among the convergences coarser than $\xi$ that are in the
class, and the corresponding operator $R$ is a functor, that is,
preserves continuity. Here are the notations we use for the corresponding
reflectors:\medskip{}

\begin{center}
\begin{tabular}{|c|c|}
\hline 
class & reflector\tabularnewline
\hline 
\hline 
pseudotopology & $\S$\tabularnewline
\hline 
paratopology & $\S_{1}$\tabularnewline
\hline 
pretopology & $\S_{0}$\tabularnewline
\hline 
hypotopology & $\S_{\wedge1}$\tabularnewline
\hline 
countable depth & $\De$\tabularnewline
\hline 
countable depth for ultrafilters & $\De^{\mathbb{U}}$\tabularnewline
\hline 
pseudotopology of countable depth & $\De\triangle\S$\tabularnewline
\hline 
pseudotopology of countable depth for ultrafilters & $\De^{\mathbb{U}}\triangle\S=\S\De^{\mathbb{U}}$\tabularnewline
\hline 
\end{tabular}\medskip{}
\par\end{center}

The understanding emerging from Figure \ref{fig:examples} is significantly
refined by analyzing the reflectors involved. Namely, we establish
(Theorem \ref{prop:S1D1}) that

\[
\S_{1}\De^{\mathbb{\mathbb{U}}}=\S_{0},
\]
so that in particular $\S_{1}\De=\S_{0}$, but $\De\S_{1}\neq\S_{0}$
(Example \ref{exa:D1S1donotcommute}).

In contrast (Theorem \ref{thm:hypocommute}) 
\[
\S_{\wedge1}\S_{1}=\S_{1}\S_{\wedge1}=\S_{0}.
\]

We also analyze when $\De=\S_{0}$. This is the case in particular
when the pseudopaving number is countable (Proposition \ref{prop:D1isS0}),
but this condition is not necessary (Example \ref{exa:D1isS0butnotcpp}).
Here is a diagram summarizing the main relations:

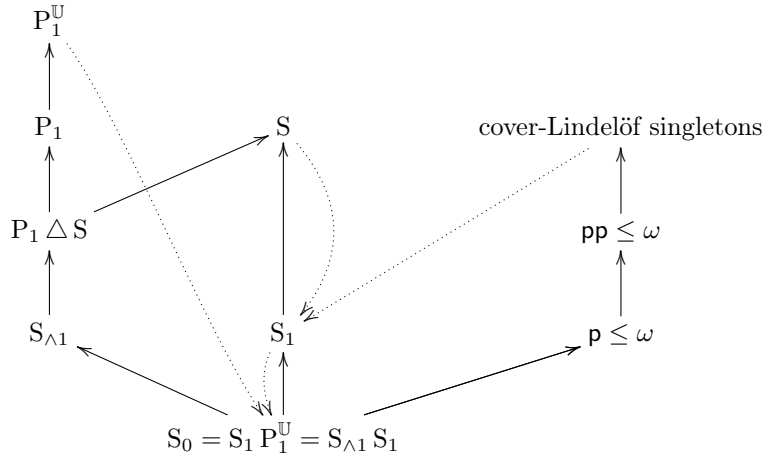
\begin{figure}[H]
\[
\xymatrix{\De^{\mathbb{U}}\ar@{.>}@(dr,ul)[ddddr]\\
\De\ar[u] & \S\ar@{.>}@(dr,ur)[dd] & \ar@{.>}[ddl]\text{cover-Lindelöf singletons}\\
\De\triangle\S\ar[u]\ar[ur] &  & \ar[u]\mathsf{pp}\leq\omega\\
\S_{\wedge1}\ar[u] & \S_{1}\ar[uu]\ar@{.>}@(dl,dr)[d] & \ar[u]\mathsf{p}\leq\omega\\
 & \S_{0}=\S_{1}\De^{\mathbb{U}}=\S_{\wedge1}\S_{1}\ar[ul]\ar[u]\ar[ur]\ar[ur]
}
\]

\medskip{}
\label{figure:reflectors}\caption{A diagram of the main relations between the classes studied.Vertical
arrows correspond to an inclusion of classes. For instance the right-bound
arrow from $\protect\S_{0}$ to $\mathsf{p}\protect\leq\omega$ indicates
that every pretopology has countable paving number. The dotted arrows
correspond to combinations of properties: A pseudotopology whose singletons
are cover-Lindelöf is a paratopology, and a paratopology that is countably
deep for ultrafilters is a pretopology.}
\end{figure}

One may note in the diagram above the condition ``cover-Lindelöf
singletons''. While definitions of compactness and its variants (countable
compactness, Lindelöf, etc) in terms of covers or in terms of filters
coincide in the topological setting, they become different notions
in the general setting of convergences. Remarkably, even singletons
do not need to be cover-compact. It turns out that the type of cover-compactness
enjoyed by singletons is relevant to our quest, and is thus analyzed
in details too. 

Finally, a very important aspect of this article is that many results
and examples are obtained using \emph{topological }characterizations
of the non-topological notions at hand in spaces of ultrafilters (with
their usual topology). It is a striking and somewhat unexpected feature
that convergence notions that only make sense for non-topological
convergence spaces turn out to find characterizations in the space
of ultrafilters in terms of classical topological properties.

More specifically, in \cite{myn.completeness} countable paving and
pseudopaving numbers of a convergence space $(X,\xi)$ were characterized
in terms of the topological properties of the subspace 
\[
\mathbb{U}_{\xi}(x)=\{\U\in\mathbb{U}X:x\in\lm_{\xi}\U\}
\]
of the space $\mathbb{U}X$ of all ultrafilters on $X$, endowed with
its usual Stone topology. It turned out \cite{myn.completeness} that
the same topological properties of the subspace $\mathbb{U}X\setminus\mathbb{U}_{\xi}X$
of non-convergent ultrafilters characterizes completeness and ultracompleteness
of $(X,\xi)$, so that known examples of complete non-ultracomplete
topological spaces easily yield examples of spaces with countable
pseudopaving number, but uncountable paving number. 

This approach is systematized in the present paper. As a result, counter-examples
to distinguish various classes of convergence spaces are produced
thanks to characterizations in terms of topological properties of
$\mathbb{U}_{\xi}(x)$, and by characterizing the same topological
property of $\mathbb{U}X\setminus\mathbb{U}_{\xi}X$. In particular,
the implications below 
\[
\text{compact}\underset{\underset{\text{Corollary \ref{cor:hemicompactnotclosed}}}{\nleftarrow}}{\then}\text{hemicompact}\underset{\underset{\mathbb{R}\setminus\mathbb{Q}}{\nleftarrow}}{\then}\sigma\text{-compact}\underset{\underset{\text{Corollary \ref{cor:Lindelofnotsigmacompact}}}{\nleftarrow}}{\then}\text{Lindelöf}\underset{\underset{\text{Corollary \ref{cor:deltaclosednotLindelof}}}{\nleftarrow}}{\then}\delta\text{-closed},
\]
cannot be reversed, even for a subspace of the form $\mathbb{U}X\setminus\mathbb{U}_{\xi}X$
of a space of ultrafilters $(\mathbb{U}X,\beta)$. Each such example
provides an $S\subset\mathbb{U}^{\circ}X$ witnessing properties that
allow for the ``semi-automatic'' production of the corresponding
examples for the non-topological local properties of a convergence
space studied here. Table \ref{table:summaryUX} at the end of Section
\ref{sec:inStone} gathers the main results in this direction. 

\section{\label{sec:Preliminaries} Preliminaries}

We use notations and terminology consistent with the recent book \cite{DM.book}.
We refer the reader to \cite{DM.book} for a comprehensive treatment
of convergence spaces.

\subsection{Set-theoretic conventions and spaces of ultrafilters}

If $X$ is a set, we denote by $\mathbb{P}X$ its powerset, by $[X]^{<\infty}$
the set of finite subsets of $X$ and by $[X]^{\omega}$ the set of
countable subsets of $X$. If $\A\subset\mathbb{P}X$, we write 
\begin{eqnarray*}
\A^{\uparrow} & := & \left\{ B\subset X:\exists A\in\A,A\subset B\right\} \\
\A^{\cap} & := & \left\{ \bigcap_{S\in\F}S:\F\in[\A]^{<\infty}\right\} \\
\A^{\cup} & := & \left\{ \bigcup_{S\in\F}S:\F\in[\A]^{<\infty}\right\} \\
\A^{\#} & := & \left\{ B\subset X:\forall A\in\A,A\cap B\neq\emptyset\right\} .
\end{eqnarray*}

We say that two subsets $\A$ and $\B$ of $\mathbb{P}X$ (henceforth
two \emph{families of subsets of }$X$) \emph{mesh}, in symbols $\A\#\B$,
if $\A\subset\B^{\#}$, equivalently, $\B\subset\A^{\#}$. 

A family $\F$ of \emph{non-empty }subsets of $X$ is called a \emph{filter
}if $\F=\F^{\cap}=\F^{\uparrow}$ and a \emph{filter-base }if $\F^{\uparrow}$
is a filter. Dually, a family $\P$ of \emph{proper }subsets of $X$
is an \emph{ideal }if $\P=\P^{\cup}=\P^{\downarrow}$ and an \emph{ideal-base
}if $\P^{\downarrow}$ is an ideal. Of course, $\P$ is an ideal if
and only if 
\[
\P_{c}:=\left\{ X\setminus P:P\in\P\right\} 
\]
 is a filter. We denote by $\mathbb{F}X$ the set of filters on $X$.
Note that $\mathbb{P}X$ is the only family $\A$ satisfying $\A=\A^{\uparrow}=\A^{\cap}$
that has an empty element and the only family satisfying $\A=\A^{\downarrow}=\A^{\cup}$
that contains $X$. Thus we sometimes call $\{\emptyset\}^{\uparrow}=\mathbb{P}X=\{X\}^{\downarrow}$
the \emph{degenerate filter on }$X$ or \emph{the degenerate ideal
on $X$}. The set $\mathbb{F}X$ is ordered by 
\begin{equation}
\F\leq\G\iff\forall F\in\F\;\exists G\in\G\;\;G\subset F.\label{eq:filterorder}
\end{equation}
Infima for this order always exists: if $\mathbb{D}\subset\mathbb{F}X$
then $\bigwedge_{\D\in\mathbb{D}}\D=\bigcap_{\D\in\mathbb{D}}\D.$
In contrast, the supremum of even a pair of filters may fail to exist.
In fact, $\F\vee\G$ exists if and only if $\F\#\G$ and is then generated
by sets of the form $F\cap G$ for $F\in\F$ and $G\in\G$. Maximal
elements of $\mathbb{F}X$ are called \emph{ultrafilters }and $\F\in\mathbb{F}X$
is an ultrafilter if and only if $\F=\F^{\#}$. 

Dually to (\ref{eq:filterorder}), we say that a family $\mathcal{R}$
of subsets of $X$ is a \emph{refinement of }another family $\P\subset\mathbb{P}X$,
in symbols, $\mathcal{R}\vartriangleleft\P$, if for every $R\in\mathcal{R}$
there is $P\in\P$ with $R\subset P$, that is,
\begin{eqnarray}
\mathcal{R}\vartriangleleft\P & \iff & \forall R\in\mathcal{R\;\exists P\in\P}\;\;R\subset P\label{eq:refine}\\
 & \iff & \P_{c}\geq\mathcal{R}_{c}.\nonumber 
\end{eqnarray}

Given a map $o:\mathbb{P}X\to\mathbb{P}X$ and $\A\subset\mathbb{P}X$,
we write 
\[
o^{\natural}\A:=\left\{ o(A):A\in\A\right\} .
\]
Note that when $o$ is \emph{expansive (i.e., $A\subset o(A)$ }for
all $A\in\mathbb{P}X$) and $\F$ is a filter-base, then $o^{\natural}\F$
is also a filter-base. Hence in the context of filters and expansive
operators $o$, we often do not distinguish between $(o^{\natural}\F)^{\uparrow}$
and $o^{\natural}\F$.

We denote by $\mathbb{U}X$ the set of ultrafilters on $X$. We call
\emph{kernel of $\F\in\mathbb{F}X$ }the set $\ker\F:=\bigcap_{F\in\F}F.$
\begin{lem}
\label{lem:kernels} If $\mathbb{D}\subset\mathbb{F}X$ then 
\[
\ker\bigwedge_{\D\in\mathbb{D}}\D=\bigcup_{\D\in\mathbb{D}}\ker\D.
\]
\end{lem}

\begin{proof}
Suppose $x\notin\bigcup_{\D\in\mathbb{D}}\ker\D$, that is, $x\notin\ker\D$
for every $\D\in\mathbb{D}$. In other words, for every $\D\in\mathbb{D}$
there is $D_{\D}\in\D$ with $x\notin D_{\D}$. Then $x\notin\bigcup_{\D\in\mathbb{D}}D_{\D}$
so that $x\notin\ker\bigwedge_{\D\in\mathbb{D}}\D$. 

Conversely, if $x\notin\ker\bigwedge_{\D\in\mathbb{D}}\D$, there
is $A\in\bigwedge_{\D\in\mathbb{D}}\D$ with $x\notin A$ so that
$x\notin\ker\D$ for all $\D\in\mathbb{D}$, that is, $x\notin\bigcup_{\D\in\mathbb{D}}\ker\D$. 
\end{proof}
We use the convention that $\mathbb{F}$ is the class of all filters,
while $\mathbb{F}X$ is the set of all filters on $X$, and similarly,
$\mathbb{U}$ is the class of ultrafilters. Thus we write $\mathbb{U}\subset\mathbb{F}$
to indicate that $\mathbb{U}X\subset\mathbb{F}X$ for every set $X$.
We will consider other classes of filters with the same conventions:
for instance, given a cardinal $\kappa$, $\mathbb{F}_{\kappa}$ denotes
the class of filters with a filter base of cardinality less than $\kappa$.
We use $\mathbb{F}_{0}$ as a shorthand for $\mathbb{F}_{\aleph_{0}}$,
that is, for the class of \emph{principal filters} (a filter $\F$
is principal if $\ker\F\in\F$). Similarly, $\mathbb{F}_{1}$ is a
shorthand for $\mathbb{F}_{\aleph_{1}}$ and denotes the class of
filters with a countable filter-base, or \emph{countably based filters}.
We will also consider the class $\mathbb{F}_{\wedge1}$ of filters
closed under countable intersections, and for each class of filter,
the exponent $^{\circ}$ denote the subclass of \emph{free filters
}(a filter $\F$ is\emph{ free} if $\ker\F=\emptyset$). Hence $\mathbb{F^{\circ}}$
is the class of all free filters, $\mathbb{U}^{\circ}$ the class
of free ultrafilters, $\mathbb{F}_{1}^{\circ}$ the class of free
countably based filters, and so on. Additionally, if $\mathbb{D}$
is a class of filters, $\mathbb{D}_{*}=\left\{ \D_{c}:\D\in\mathbb{D}\right\} $
is the dual class of ideals. 

An ultrafilter $\U$ on $X$ is either principal (and then of the
form $\{x\}^{\uparrow}$ for some $x\in X$) or free. The map $j:X\to\mathbb{U}X$
defined by $j(x)=\{x\}^{\uparrow}$ is one-to-one, and we often identify
$X$ with $j[X]$, hence writing, $\mathbb{U}^{\circ}X=\mathbb{U}X\setminus X$. 

Given a map $f:X\to Y$ and a filter $\F\in\mathbb{F}X$, we define
the \emph{image filter }(\footnote{Alternatively, this filter can be described as $f[\F]=\{B\subset Y:f^{-1}[B]\in\F\}$,
which is a filter, while $\{f(F):F\in\F\}$ is only a filter-base.})\emph{ }
\[
f[\F]=\left\{ f(F):F\in\F\right\} ^{\uparrow}\in\mathbb{F}Y.
\]

If $\F\in\mathbb{U}X$ then $f[\F]\in\mathbb{U}Y$, so that a map
$f:X\to Y$ induces a map $\mathbb{U}f:\mathbb{U}X\to\mathbb{U}Y$
defined by $\mathbb{U}f(\U)=f[\U]$. 

If $R\subset X\times Y$ is a relation, and $x\in X$, we denote by
$R(x)=\{y\in Y:(x,y)\in R\}$, and if $F\subset X$, we write $R[F]=\bigcup_{x\in F}R(x)$.
If $\F\in\mathbb{F}X$ then
\[
R[\F]=\{R[F]:F\in\F\}^{\uparrow_{Y}}
\]
 is a (possibly degenerate) filter on $Y$. We say that a class $\mathbb{D}$
of filters is $\mathbb{F}_{0}$-composable if for every $X,Y$ and
$R\subset X\times Y$,
\[
\D\in\mathbb{D}X\then R[\D]\in\mathbb{D}Y,
\]
with the convention that the degenerate filter of a set $Z$ is always
in $\mathbb{D}Z$.

If $R\subset X\times Y$ then $R^{-}\subset Y\times X$ denotes the
inverse relation. If $\F$ is a family of subsets of $X$ and $\G$
is a family of subsets of $Y$ then
\begin{equation}
R\in(\F\times\G)^{\#}\iff R[\F]\#\G\iff R^{-}[\G]\#\F.\label{eq:meshprod}
\end{equation}
 If $\F\subset\mathbb{P}A$ and $\G:A\to\mathbb{P}(\mathbb{P}X)$,
we define the\emph{ contour of $\G$ along} $\F$ to be the family
of subsets of $X$ given by
\begin{equation}
\G(\F)=\bigcup_{a\in F}\bigcap_{F\in\F}\G(a).\label{eq:contour}
\end{equation}
Note that if $\F\in\mathbb{F}A$ and $\G:A\to\mathbb{F}X$, then $\G(\F)\in\mathbb{F}X$,
and if $\F\in\mathbb{U}A$ and $\G:A\to\mathbb{U}X$, then $\G(\F)\in\mathbb{U}X$.

We endow $\mathbb{U}X$ with its usual Stone topology, here denoted
$\beta$, given by the base $\left\{ \beta A:A\subset X\right\} $,
where 
\[
\beta A=\left\{ \U\in\mathbb{U}X:A\in\U\right\} .
\]
This topology is zero-dimensional (because the sets $\beta A$ are
clopen) and compact Hausdorff. In particular, every ultrafilter $\mathfrak{U}\in\mathbb{U}(\mathbb{U}X)$
converges for $\beta$ to a unique element of $\mathbb{U}X$:
\[
\lm_{\beta}\mathfrak{U}=\{\mathfrak{U}^{\star}\},
\]
where 
\[
\mathfrak{U}^{\star}:=\bigcup_{\mathbb{A}\in\mathfrak{U}}\bigcap_{\W\in\mathbb{A}}\W=\mathcal{I}(\mathfrak{U})
\]
 is the contour\emph{ }of the identity map $\mathcal{I}$ of $\mathbb{U}X$
along $\mathfrak{U}$. 

If $\H\in\mathbb{F}X$, let 
\[
\beta\H=\{\U\in\mathbb{U}X:\H\leq\U\}.
\]
With this convention, $\beta A$ is a shorthand for $\beta\{A\}^{\uparrow}$.
A non-empty subset $H$ of $\mathbb{U}X$ is closed for $\beta$ if
and only if there is a filter $\H$ on $X$ with $H=\beta\H$. If
$H\subset\mathbb{U}X$, let $\cl_{\beta}H=\beta\left(\bigwedge_{\U\in H}\U\right)$
denote the closure in the Stone topology of $\mathbb{U}X$ and let
$\cl_{\delta}H$ denote the closure in the associated $G_{\delta}$-topology.
Observe that:
\begin{lem}
\label{lem:GdeltaBase} The family
\[
\left\{ \beta\H:\H\in\mathbb{F}_{1}\right\} 
\]
form a basis for the $G_{\delta}$-topology of $(\mathbb{U}X,\beta)$.
\end{lem}

\begin{proof}
If $\H\in\mathbb{F}_{1}$ and $(H_{n})_{n\in\omega}$ is a filter-base
for $\H$ then $\beta\H=\bigcap_{n\in\omega}\beta H_{n}$ so that
$\beta\H$ is $G_{\delta}$ (and closed). Let $\mathbb{A}$ be a $G_{\delta}$-subset
of $\mathbb{U}X$. Then there are open subsets $\mathbb{O}_{i}$ of
$\mathbb{U}X$ with $\mathbb{A}=\bigcap_{i\in\omega}\mathbb{O}_{i}$
and for each $i$, there is $\A_{i}\subset\mathbb{P}X$ with $\mathbb{O}_{i}=\bigcup_{A\in\A_{i}}\beta A$.
Let $\U\in\mathbb{A}$. Then for each $i\in\omega$, there is $A_{i}\in\A_{i}$
with $\U\in\beta A_{i}$ so that 
\[
\U\in\bigcap_{i\in\omega}\beta A_{i}\subset\bigcap_{i\in\omega}\mathbb{O}_{i}\subset\mathbb{A}.
\]
Moreover, $\bigcap_{i\in\omega}\beta A_{i}$ is a $G_{\delta}$-subset
of $\mathbb{U}X$ and $\bigcap_{i\in\omega}\beta A_{i}\neq\emptyset$
ensures that $\{A_{i}:\in\omega\}$ has the finite intersection property
and generates a filter $\H\in\mathbb{F}_{1}$ with 
\[
\beta\H=\bigcap_{i\in\omega}\beta A_{i}.
\]

Hence every $G_{\delta}$-set is a union of $G_{\delta}$-sets of
the form $\beta\H$ for $\H\in\mathbb{F}_{1}$ and thus every $\delta$-open
set is also a union of such sets.
\end{proof}
Note also that if $\mathbb{H}\subset\mathbb{F}X$ is a family of filters
on $X$, then
\begin{equation}
\beta\Big(\bigwedge_{\H\in\mathbb{H}}\H\Big)=\cl_{\beta}\Big(\bigcup_{\H\in\mathbb{H}}\beta\H\Big),\label{eq:betainf}
\end{equation}
because 
\[
\beta\Big(\bigwedge_{\H\in\mathbb{H}}\H\Big)=\beta\Big(\bigwedge_{\H\in\mathbb{H}}\bigwedge_{\U\in\beta\H}\U\Big)=\beta\Big(\bigwedge_{\U\in\bigcup_{\H\in\mathbb{H}}\beta\H}\U\Big)=\cl_{\beta}\Big(\bigcup_{\H\in\mathbb{H}}\beta\H\Big).
\]

Since for $f:X\to Y$, $B\subset Y$, and $\F\in\mathbb{F}X$, 
\[
(\mathbb{U}f)^{-1}\left(\beta B\right)=\beta\left(f^{-1}[B]\right)
\]
and 
\[
\mathbb{U}f[\beta\F]=\beta\left(f[\F]\right),
\]
by \cite[Corollary II.6.7]{DM.book}, we conclude that
\begin{lem}
\label{lem:Uf} If $f:X\to Y$ then $\mathbb{U}f:\mathbb{U}X\to\mathbb{U}Y$
is a continuous perfect map.
\end{lem}

\subsection{Convergences}

Recall that a convergence $\xi$ on a set $X$ is a relation between
$X$ and the set $\mathbb{F}X$ of filters on $X$ (that is, $\xi\subset X\times\mathbb{F}X$),
satisfying $(x,\{x\}^{\uparrow})\in\xi$ for all $x\in X$ and 
\[
\F\leq\G\then\xi^{-}(\F)\subset\xi^{-}(\G)
\]
for every filters $\F,\G\in\mathbb{F}X$. We interpret $(x,\F)\in\xi$
as $\F$ \emph{converges to $x$ }for $\xi$ and hence we use the
alternative notations
\[
x\in\lm_{\xi}\F\iff\F\in\lm_{\xi}^{-}(x)
\]
for $(x,\F)\in\xi$. In those terms, the two requirements on $\xi\subset X\times\mathbb{F}X$
to be a convergence become
\begin{gather}
x\in\lm_{\xi}\{x\}^{\uparrow}\label{eq:centered}\\
\F\leq\G\then\lm_{\xi}\F\subset\lm_{\xi}\G,
\end{gather}
for every $x\in X$ and filters $\F,\G\in\mathbb{F}X$.

A map $f$ between two convergence spaces $(X,\xi)$ and $(Y,\sigma)$
is \emph{continuous} if for every $\F\in\mathbb{F}X$ and $x\in X$,
\[
x\in\lm_{\xi}\F\then f(x)\in\lm_{\sigma}f[\F].
\]

Consistently with \cite{DM.book}, we denote by $|\xi|$ the underlying
set of a convergence $\xi$, and, if $(X,\xi)$ and $(Y,\sigma)$
are two convergence spaces we often write $f:|\xi|\to|\sigma|$ instead
of $f:X\to Y$ even though one may see it as improper since many different
convergences have the same underlying set. This allows to talk about
the continuity of $f:|\xi|\to|\sigma|$ without having to repeat for
what structure. 

Every topology can be seen as a convergence. Indeed, if $\tau$ is
a topology and $\N_{\tau}(x)$ denotes the neighborhood filter of
$x$ for $\tau$, then 
\[
x\in\lim\F\iff\F\geq\N_{\tau}(x)
\]
defines a convergence that completely characterizes $\tau$. Moreover,
a function between two topological space is continuous in the topological
sense if and only if it is continuous in the sense of convergences.
Hence, we do not distinguish between a topology $\tau$ and the convergence
it induces, thus embedding the category $\mathbf{Top}$ of topological
spaces and continuous maps as a full subcategory of the category $\mathbf{Conv}$
of convergence spaces and continuous maps.

A point $x$ of a convergence space $(X,\xi)$ is \emph{isolated }if
$\lim_{\xi}^{-}(x)=\{\{x\}^{\uparrow}\}$. A \emph{prime }convergence
is a convergence with at most one non-isolated point.

Given two convergences $\xi$ and $\theta$ on the same set $X$,
we say that $\xi$ is \emph{finer than $\theta$ }or that $\theta$
is \emph{coarser than }$\xi$, in symbols $\xi\geq\theta$, if the
identity map from $(X,\xi)$ to $(X,\theta)$ is continuous, that
is, if $\lim_{\xi}\F\subset\lim_{\theta}\F$ for every $\F\in\mathbb{F}X$.
With this order, the set $\C(X)$ of convergences on $X$ is a complete
lattice whose greatest element is the discrete topology and least
element is the antidiscrete topology, and for which, given $\Xi\subset\C(X)$,
\[
\lm_{\bigvee\Xi}\F=\bigcap_{\xi\in\Xi}\lm_{\xi}\F\text{ and }\lm_{\bigwedge\Xi}\F=\bigcup_{\xi\in\Xi}\lm_{\xi}\F.
\]
In fact, $\mathbf{Top}$ is a reflective subcategory of $\mathbf{Conv}$
and the corresponding reflector $\T$, called \emph{topologizer},
associates to each convergence $\xi$ on $X$ its \emph{topological
modification $\T\xi$, }which is the finest topology on $X$ among
those coarser than $\xi$ (in $\mathbf{Conv}$). Concretely, $\T\xi$
is the topology whose closed sets are the subsets of $|\xi|$ that
are $\xi$-\emph{closed, }that is, subsets $C$ satisfying
\[
C\in\F^{\#}\then\lm_{\xi}\F\subset C.
\]
 A subset $O$ of $|\xi|$ is $\xi$-\emph{open} if its complement
is closed, equivalently if 
\[
\lm_{\xi}\F\cap O\neq\emptyset\then O\in\F.
\]

Given a map $f:|\xi|\to Y$, there is the finest convergence $f\xi$
on $Y$ making $f$ continuous (from $\xi$), and given $f:X\to|\sigma|$,
there is the coarsest convergence $f^{-}\sigma$ on $X$ making $f$
continuous (to $\sigma$). The convergences $f\xi$ and $f^{-}\sigma$
are called \emph{final convergence for $f$ and $\xi$ }and \emph{initial
convergence for $f$ and $\sigma$ }respectively. Note that 
\begin{equation}
f:|\xi|\to|\sigma|\text{ is continuous }\iff\xi\geq f^{-}\sigma\iff f\xi\geq\sigma.\label{eq:continuity}
\end{equation}

If $A\subset|\xi|$, the \emph{induced convergence by $\xi$ on $A$,
}or \emph{subspace convergence},\emph{ }is $i^{-}\xi$, where $i:A\to|\xi|$
is the inclusion map. If $\xi$ and $\tau$ are two convergences,
the \emph{product convergence $\xi\times\tau$ }on $|\xi|\times|\tau|$
is the coarsest convergence on $|\xi|\times|\tau|$ making both projections
continuous, that is,
\[
\xi\times\tau:=p_{\xi}^{-}\xi\vee p_{\tau}^{-}\tau,
\]
where $p_{\xi}:|\xi|\times|\tau|\to|\xi|$ and $p_{\tau}:|\xi|\times|\tau|\to|\tau|$
are the projections defined by $p_{\xi}(x,y)=x$ and $p_{\tau}(x,y)=y$
respectively.

The \emph{adherence }of a filter $\H$ on a convergence space $(X,\xi)$
is 
\begin{equation}
\adh_{\xi}\H=\bigcup_{\F\#\H}\lm_{\xi}\F=\bigcup_{\U\in\beta\H}\lm_{\xi}\U,\label{eq:adherence}
\end{equation}
so that $\adh_{\xi}\U=\lim_{\xi}\U$ for every ultrafilter $\U$.

Given a class $\mathbb{D}$ of filters, a convergence $\xi$ on $X$
is \emph{determined by adherences of }$\mathbb{D}$-\emph{filters}
if 
\[
\bigcap_{\mathbb{D}X\ni\D\#\F}\adh_{\xi}\D\subset\lm_{\xi}\F,
\]
where the reverse inclusion is always true. 
\begin{prop}
\cite[Corollary XIV.3.8]{DM.book} If $\mathbb{D}$ is an $\mathbb{F}_{0}$-composable
class of filters, then the class of convergence spaces determined
by adherences of $\mathbb{D}$-filters is concretely reflective and
the reflector $\AdhD$ is characterized by 
\[
\lm_{\AdhD\xi}\F=\bigcap_{\mathbb{D}\ni\D\#\F}\adh_{\xi}\D,
\]
and satisfies
\[
\AdhD(f^{-}\tau)=f^{-}\left(\AdhD\tau\right)
\]
for every map $f$ and convergence $\tau$ on its codomain.
\end{prop}

The classes $\mathbb{F}$, $\mathbb{F}_{1}$, $\mathbb{F}_{\wedge1}$
and $\mathbb{F}_{0}$ are all $\mathbb{F}_{0}$-composable, and we
will use the following terminology and notations:\bigskip{}

\begin{center}
\begin{tabular}{|c|c|c|}
\hline 
$\mathbb{D}$ & reflector $\AdhD$ & a convergence $\xi=\AdhD\xi$ is called a\tabularnewline
\hline 
\hline 
$\mathbb{F}$ & $\S$ & pseudotopology\tabularnewline
\hline 
$\mathbb{F}_{1}$ & $\S_{1}$ & paratopology\tabularnewline
\hline 
$\mathbb{F}_{0}$ & $\S_{0}$ & pretopology\tabularnewline
\hline 
$\mathbb{F}_{\wedge1}$ & $\S_{\wedge1}$ & hypotopology\tabularnewline
\hline 
\end{tabular}\bigskip{}
\par\end{center}

Of course, every pretopology is a paratopology and a hypotopology,
while every paratopology and every hypotopology is a pseudotopology.
Note that a convergence $\xi$ is a topology if and only if it is
a pretopology and the adherence operator restricted to principal filters
is idempotent (i.e., $\adh_{\xi}(\adh_{\xi}A)=\adh_{\xi}A$ for all
$A\subset|\xi|$), in which case it coincides with the topological
closure, denoted $\cl_{\xi}$.

Accordingly
\[
\T\leq\S_{0}\leq\S_{1}\leq\S\text{ and }\T\leq\S_{0}\leq\S_{\wedge1}\leq\S.
\]

Given a convergence $\xi$ and $x\in|\xi|$, the filter 
\begin{equation}
\V_{\xi}(x):=\bigwedge\left\{ \F:x\in\lm_{\xi}\F\right\} \label{eq:vicinity}
\end{equation}
is called the \emph{vicinity filter of} $x$ for $\xi$. In general,
$\V_{\xi}(x)$ does not need to converge to $x$ for $\xi$. It is
a simple exercise to check that $\xi$ is a pretopology if and only
if $x\in\lim_{\xi}\V_{\xi}(x)$ for each $x\in|\xi|$, and that 
\[
\V_{\xi}(x)=\V_{\S_{0}\xi}(x)
\]
 for all $x\in|\xi|$. 

Note that $\V_{\xi}(x):X\to\mathbb{F}X$ so that given $\A\subset\mathbb{P}X$
we may consider the contour 
\[
\V_{\xi}(\A)=\bigcup_{A\in\A}\bigcap_{x\in A}\V_{\xi}(x)
\]
of $\V_{\xi}(x)$ along $\A$, which satisfies
\begin{equation}
\B\#\V_{\xi}(\A)\iff\adh_{\xi}^{\natural}\B\#\A.\label{eq:Vtoadh}
\end{equation}

Given a class $\mathbb{D}$ of filters, we call a convergence $\mathbb{D}$-\emph{based}
if whenever $x\in\lim_{\xi}\F$, there is $\D\in\mathbb{D}|\xi|$
with $\D\leq\F$ and $x\in\lm_{\xi}\D$.
\begin{prop}
\cite[Corollary XIV.4.4]{DM.book} If $\mathbb{D}$ is an $\mathbb{F}_{0}$-composable
class of filters then the class of $\mathbb{D}$-based convergences
is (concretely) coreflective and the coreflector is given by
\[
\lm_{\BaseD\xi}\F=\bigcup_{\mathbb{D}\ni\D\leq\F}\lm_{\xi}\D.
\]
\end{prop}

We will be particularly interested in the case of $\mathbb{F}_{1}$-based
convergences, or \emph{first-countable }convergences, in which case
we use the notation $\I_{1}$ for $\operatorname{B}_{\mathbb{F}_{1}}$.

It turns out (see \cite{dolecki1996convergence} or \cite[Section XIV.6]{DM.book}
for details) that for a topology $\xi$,
\begin{eqnarray*}
\xi\text{ sequential} & \iff & \xi\geq\T\I_{1}\xi\\
\xi\text{ Fréchet-Urysohn} & \iff & \xi\geq\S_{0}\I_{1}\xi\\
\xi\text{ strongly Fréchet} & \iff & \xi\geq\S_{1}\I_{1}\xi\\
\xi\text{ bisequential} & \iff & \xi\geq\S\I_{1}\xi\\
\xi\text{ first-countable} & \iff & \xi\geq\I_{1}\xi,
\end{eqnarray*}
and we take these functorial inequalities as the definitions of these
notions for general convergences.

A convergence is called \emph{Hausdorff }if the cardinality of $\lim\F$
is at most one, for every filter $\F$. Of course, a topology is Hausdorff
in the usual topological sense if and only if it is in the convergence
sense. 

Given two convergences $\xi$ and $\theta$ on the same set, we say
that $\xi$ is $\theta$-\emph{regular} if 
\[
\lm_{\xi}\F\subset\lm_{\xi}\adh_{\theta}^{\natural}\F
\]
for every filter $\F$. If $\theta=\xi$, the convergence is simply
called \emph{regular }and if $\theta=\T\xi$, then we say that $\xi$
is $\T$-\emph{regular}. A topology is regular in the usual sense
if and only if it is regular (or of course $\T$-regular) in the convergence
sense.

\subsection{Compactness}

A primary source for complements on compactness and its variants in
the context of convergence spaces is \cite{D.comp}. 

\subsubsection{Filter compactness}

A subset $K$ of a convergence space $(X,\xi)$ is \emph{compact }if
\[
\forall\F\in\mathbb{F}X\;K\in\F^{\#}\then\adh_{\xi}\F\cap K\neq\emptyset.
\]

More generally, given a class $\mathbb{D}$ of filters, a family $\A$
of subsets of $|\xi|$ is $\mathbb{D}$-\emph{compact} at another
family $\B$ if
\[
\forall\D\in\mathbb{D}X\;\D\#\A\then\adh_{\xi}\D\in\B^{\#}.
\]
A family of subsets of $(X,\xi)$ that is $\mathbb{D}$-compact at
itself is simply called $\mathbb{D}$-\emph{compact}. It is called
\emph{$\mathbb{D}$-compactoid }if it is $\mathbb{D}$-compact at
$\{X\}$\emph{. }In the case $\mathbb{D}=\mathbb{F}$, we omit the
class of filters, so that a subset $K$ of a convergence space is
compact(oid) if and only if its principal filter $\{K\}^{\uparrow}$
is a compact(oid) filter. Let $\mathbb{K}_{\mathbb{D}}$ denote the
class of $\mathbb{D}$-compactoid filters.

A convergence $\xi$ is \emph{locally compact(oid) }if every convergent
filter has a compact(oid) element. 

Let $\mathcal{K}(\xi)$ or simply $\mathcal{K}$ denote the set of
compactoid subsets of $|\xi|.$ Note that unless $\xi$ is compact,
the family $\mathcal{K}_{c}$ of complements of compactoid sets is
a filter-base, generating the \emph{cocompactoid filter}.
\begin{lem}
\label{lem:consonantcovers} Let $\D\in\text{\ensuremath{\mathbb{D=\cl^{\natural}\mathbb{D}}\subset}}\mathbb{F}X$
and let $\xi$ be a non-compact convergence. Then 
\[
\adh_{\T\xi}\D=\emptyset\iff\cl^{\natural}\D\geq(\mathcal{K}_{\mathbb{D}})_{c}.
\]
\end{lem}

\begin{proof}
If 
\[
\adh_{\xi}\D\subset\adh_{\T\xi}\D=\emptyset=\bigcap_{D\in\D}\cl_{\xi}D,
\]
then $\adh_{\xi}\cl^{\natural}\D=\emptyset$. Thus $K\notin(\cl^{\natural}\D)^{\#}$
for every $K\in\mathcal{K}_{\mathbb{D}}(\xi)$ because $\cl^{\natural}\D\in\mathbb{D}$;
in other words, $K^{c}\in\cl^{\natural}\D$, that is, $\cl^{\natural}\D\geq(\mathcal{K}_{\mathbb{D}})_{c}$.
Assume conversely that $\cl^{\natural}\D\geq(\mathcal{K}_{\mathbb{D}})_{c}$.
For every $x\in|\xi|$, $\{x\}^{c}\in\cl^{\natural}\D$ because $\{x\}\in\mathcal{K}$.
Hence $\{x\}\notin(\cl^{\natural}\D)^{\#}$, that is, there is $U\in\O(x)$
with $U\notin\D^{\#}$. As a result, $x\notin\adh_{\T\xi}\D$.
\end{proof}
Recall from \cite{DM.book} that a filter $\F$ on a convergence space
$(X,\xi)$ is called \emph{closed }if $\adh_{\xi}\F=\ker\F$ and that
a filter on a Hausdorff convergence space is compact if and only if
it is closed and compactoid \cite[Corollary IX.7.15]{DM.book}. 
\begin{prop}
\label{prop:compactoidtocompact} If $\xi$ is $\theta$-regular and
$\F$ is a filter on $|\xi|$ then $\F$ is compactoid if and only
if $\adh_{\theta}^{\natural}\F$ is compactoid. 
\end{prop}

\begin{proof}
If $\adh_{\theta}^{\natural}\F$ is compactoid so is any finer filter,
in particular $\F$. Assume conversely that $\F$ is compactoid and
let $\H\#\adh_{\theta}^{\natural}\F$, equivalently(\footnote{Recall the definition (\ref{eq:contour}) of contours.}),
$\V_{\theta}(\H)\#\F$ by \cite[Prop. VI.5.2]{DM.book}. Thus $\adh_{\xi}\V_{\theta}(\H)\neq\emptyset$.
In view of \cite[Prop. VIII.4.1]{DM.book}, $\adh_{\xi}\H\neq\emptyset$.
Thus $\adh_{\theta}^{\natural}\F$ is compactoid.
\end{proof}
\begin{cor}
\label{cor:compactadherence} If $\xi$ is a $\T$-regular pseudotopology
and $\F$ is a filter on $|\xi|$, the following are equivalent:
\begin{enumerate}
\item $\F$ is compactoid;
\item $\cl_{\xi}^{\natural}\F$ is compactoid;
\item $\cl_{\xi}^{\natural}\F$ is compact.
\end{enumerate}
In this case, $\adh_{\xi}(\cl_{\xi}^{\natural}\F)=\ker\cl_{\xi}^{\natural}\F$
is a non-empty compact set.
\end{cor}

\begin{proof}
The equivalence of (1) and (2) is Proposition \ref{prop:compactoidtocompact}
for $\theta=\T\xi$. By definition, (3) implies (2), so we only need
to prove that (2) implies (3).

To this end, let $\U\in\beta(\cl_{\xi}^{\natural}\F).$ By compactoidness,
$\lim_{\xi}\U\neq\emptyset$. Moreover, if $x\in\lim_{\xi}\U\subset\lim_{\T\xi}\U$,
we have $x\in\lim_{\T\xi}\N_{\xi}(\U)$. Additionally $\N_{\xi}(\U)\#\F$
because $\U\#\cl_{\xi}^{\natural}\F$, so that 
\[
x\in\bigcap_{F\in\F}\cl_{\xi}F=\ker(\cl_{\xi}^{\natural}\F).
\]
Thus $\lim_{\xi}\U\cap\ker(\cl_{\xi}^{\natural}\F)\neq\emptyset$
and $\cl_{\xi}^{\natural}\F$ is a compact filter. By \cite[Cor. IX.7.16]{DM.book},
$\adh_{\xi}(\cl_{\xi}^{\natural}\F)=\ker\cl_{\xi}^{\natural}\F$ is
a non-empty compact set.
\end{proof}
Note also that
\begin{prop}
\label{prop:compactfiltertoset} If $\xi$ is a topology then $\D$
is compact(oid) if and only if $\O_{\xi}(\D)$ is.
\end{prop}

\begin{proof}
Assume that $\D$ is compact(oid). Let $\H\#\O(\D)$ equivalently,
$\cl^{\natural}\H\#\D$. By compactness of $\D$ (resp. compactoidness)
we have 
\[
\adh_{\xi}\cl_{\xi}^{\natural}\H\underset{\xi=\T\xi}{=}\adh_{\xi}\H
\]
 meshes with $\D$, hence with $\O(\D)$ (resp. is non-empty).

Conversely assume $\O(\D)$ is compact. Since $\D\geq\O(\D)$ if $\H\#\D$
then $\H\#\O(\D)$ and thus $\adh\H\#\O(\D)$, equivalently, $\cl(\adh\H)\#\D$
but $\cl(\adh\H)=\adh\H$.
\end{proof}
Moreover, \cite[Theorem IX.7.17]{DM.book} states that for a Hausdorff
convergence $\xi$ if a filter $\D=\O_{\xi}(\D)^{\uparrow}$ is compact
then $\D=\O_{\xi}(K)^{\uparrow}$ where $K=\adh\D$ is compact. 
\begin{cor}
\label{cor:O(compact)} If $\xi$ is a regular Hausdorff topology
and $\D$ is a compact filter then $\adh\D$ is compact and
\[
\O(\D)^{\uparrow}=\O(\adh\D)^{\uparrow}.
\]
\end{cor}

\begin{proof}
In view of Proposition \ref{prop:compactfiltertoset}, $\O(\D)^{\uparrow}$
is a compact filter and thus by \cite[Theorem IX.7.17]{DM.book},
$\O(\D)^{\uparrow}=\O(\adh\O(\D)^{\uparrow})^{\uparrow}$ where $\adh\O(\D)^{\uparrow}$
is compact. Moreover, as $\xi$ is a regular topology, $\H\#\O(\D)$
if and only if $\cl^{\natural}\H\#\D$ and $\lim_{\xi}\H=\lim_{\xi}\cl^{\natural}\H$
so that $\adh\O(\D)=\adh\D$, and the result follows.
\end{proof}
Recall \cite{quest} that a topology is \emph{bi-k }if every convergent
ultrafilter there is a compact set $K$ and $\H\in\mathbb{F}_{1}$
with $\ker\H=K$ and
\[
\U\geq\H\geq\O(K)^{\uparrow}.
\]
A topology is of \emph{pointwise countable type }if each point belongs
to a compact set of countable character, that is, a compact set $K$
where $\O(K)^{\uparrow}\in\mathbb{F}_{1}$.

Corollary \ref{cor:PCTandbik} below is due to \cite{dolecki1996convergence},
but we include the proof because of partially conflicting terminology
and for thoroughness, as \cite{dolecki1996convergence} only includes
a sketch of proof.
\begin{cor}
\label{cor:PCTandbik}\cite{dolecki1996convergence} A regular Hausdorff
topology is bi-k (resp. of pointwise countable type) if and only if
every convergent ultrafilter (resp. filter) contains a countably based
compactoid filter.
\end{cor}

\begin{proof}
If $(X,\xi)$ is bi-$k$ then $\H$ witnessing the definition is a
countably based and compactoid filter. Conversely, if $\U$ is a convergent
filter, then there is a compactoid filter $\H\in\mathbb{F}_{1}$ with
$\U\geq\H$. Then $\cl^{\natural}\H$ is countably based and compact
by Corollary \ref{cor:compactadherence}, and by Corollary \ref{cor:O(compact)},
$K=\adh\H=\ker\cl^{\natural}\H$ is compact and $\O(\cl^{\natural}\H)=\O(K)$
so that $\U\geq\H\geq$ $\O(\cl^{\natural}\H)=\O(K)$ and $\xi$ is
bi-$k$.

Suppose that every convergent filter contains a countably based compactoid
filter. In particular for every $x\in X$ we have $\O(x)\geq\D$ for
some $\D\in\mathbb{K}_{\mathbb{F}}\cap\mathbb{F}_{1}$. Moreover,
in view of Corollary \ref{cor:compactadherence} and Corollary \ref{cor:O(compact)},
$K=\adh\D$ is compact and $\O(\cl^{\natural}\D)=\O(K)$. Hence
\[
\O(x)\geq\D\geq\cl^{\natural}\D\geq\O(\cl^{\natural}\D)=\O(K),
\]
so that, letting $\{D_{n}:n\in\omega\}$ be a decreasing countable
filter base of $\D$, for every $U\in\O(K)$, there is $n\in\omega$
with $K\subset\cl D_{n}\subset U$ and thus there is $O_{n}\in\O(x)$
with $K\subset O_{n}\subset U$ and $K$ is of countable character.

Conversely, assume that $X$ is a regular Hausdorff topological space
of pointwise countable type and let $\F$ be convergent, that is,
$\F\geq\O(x)$ for some $x$. Then there is $K$ of countable character
with $x\in K$ so that $\F\geq\O(x)\geq\O(K)$ and $\O(K)\in\mathbb{K}_{\mathbb{F}}\cap\mathbb{F}_{1}$. 
\end{proof}

\subsubsection{Cover-compactness}

A family $\P$ of subsets of $\xi$ is a \emph{$\xi$-cover of} $A\subset|\xi|$
if $\P\cap\F\neq\emptyset$ for every filter $\F$ with $\lim_{\xi}\F\cap A\neq\emptyset$.
Note that if $\xi$ is a topology then $\P\subset\O_{\xi}$ is a cover
if and only if $A\subset\bigcup_{P\in\P}P$, and moreover, every cover
has a refinement that is an open cover.

Dolecki noted \cite{D.comp} that the notion of cover is dual to that
of family of empty adherence. More specifically, $\P$ is a $\xi$-cover
of $A$ if and only if 
\[
\adh_{\xi}\P_{c}\cap A=\emptyset,
\]
extending the notion of adherence from filters to general families
via
\[
\adh_{\xi}\A=\bigcup_{\H\#\A}\lm_{\xi}\H.
\]
The notion can be extended to families of subsets. We say that $\P$
is a $\xi$-cover of $\A\subset2^{|\xi|}$ if $\P$ is a $\xi$-cover
of some $A\in\A$. With this in mind, 
\begin{prop}
\cite[Theorem 3.1]{D.comp} Let $\xi$ be a convergence and $\A,\P\subset2^{|\xi|}$.
Then $\P$ is a $\xi$-cover of $\A$ if and only if $\adh_{\xi}\P_{c}\notin\A^{\#}$.
\end{prop}

\begin{defn}
\label{def:covercompactness}\cite{D.comp} Let $\mathbb{D}$ and
$\mathbb{J}$ be two classes of filters and $\mathbb{D}_{*}$ and
$\mathbb{J}_{*}$ the corresponding families of dual ideals. We say
that $\A$ is \emph{cover-$\nicefrac{\mathbb{J}_{*}}{\mathbb{D}_{*}}$-compact
at} $\B$ (for a convergence $\xi$) if for every $\xi$-cover $\C\in\mathbb{J}_{*}$
of $\B$, there is $\mathcal{S}\subset\C$ with $\S\in\mathbb{D}_{*}$
such that $\mathcal{S}$ is a $\xi$-cover of $\A$, equivalently,
\[
\forall\G\in\mathbb{J}\;\adh_{\xi}\G\neg\#\B\then\exists\D\in\mathbb{D},\D\leq\G:\adh_{\xi}\D\neg\#\A,
\]
equivalently
\[
\forall\G\in\mathbb{J}\left(\forall\D\in\mathbb{D},\D\leq\G\;\adh_{\xi}\D\in\A^{\#}\then\adh_{\xi}\G\in\B^{\#}\right).
\]
\end{defn}

\begin{rem}
\label{rem:refinement} In view of (\ref{eq:refine}) $\A$ is cover-$\nicefrac{\mathbb{J}_{*}}{\mathbb{D}_{*}}$-compact
at $\B$ if and only if every cover in $\mathbb{J}_{*}$ of $\B$
has a refinement in $\mathbb{D}_{*}$ that is a cover of $\A$.
\end{rem}

\begin{prop}
\cite[Proposition IX.11.13]{DM.book}\label{prop:covercompact} Let
$\xi$ be a convergence, and let $\A$ and $\B$ be two families of
subsets of $|\xi|$. $\A$ is $\xi$-cover-$\nicefrac{\mathbb{F}_{*}}{\mathbb{F}_{0*}}$-compact
at $\B$ if and only if $\V_{\xi}(\A)$ is $\xi$-compact at $\B$. 
\end{prop}

Proposition \ref{prop:covercompact} generalizes to $\nicefrac{\mathbb{J_{*}}}{\mathbb{F}_{0*}}$-cover-compactness:
\begin{prop}
Let $\xi$ be a convergence, let $\A$ and $\B$ be two families of
subsets of $|\xi|$ and let $\mathbb{J}$ be a class of filters. $\A$
is $\xi$-cover- $\nicefrac{\mathbb{J_{*}}}{\mathbb{F}_{0*}}$-compact
at $\B$ if and only if $\V_{\xi}(\A)$ is $\xi$-$\mathbb{J}$-compact
at $\B$. 
\end{prop}

\begin{proof}
Let $\A$ be $\xi$-cover- $\nicefrac{\mathbb{J_{*}}}{\mathbb{F}_{0*}}$-compact
at $\B$ and let $\G\in\mathbb{J}$ with $\G\#\V_{\xi}(\A)$, equivalently,
$\adh_{\xi}^{\natural}\G\#\A$ by (\ref{eq:Vtoadh}), so that $\adh_{\xi}\G\in\B^{\#}$.
Hence $\V_{\xi}(\A)$ is $\xi$-$\mathbb{J}$-compact at $\B$. Assume
conversely, that $\V_{\xi}(\A)$ is $\xi$-$\mathbb{J}$-compact at
$\B$ and let $\G\in\mathbb{J}$ with $\adh_{\xi}^{\natural}\G\#\A$,
equivalently, $\G\#\V_{\xi}(\A)$, so that $\adh_{\xi}\G\in\B^{\#}$,
that is, $\A$ is $\xi$-cover- $\nicefrac{\mathbb{J_{*}}}{\mathbb{F}_{0*}}$-compact
at $\B$.
\end{proof}
In that case $\A=\B=\{\{x\}\}$, we thus obtain 
\begin{cor}
\label{cor:compactpoints} If $\xi$ is a convergence and $\mathbb{J}$
is a class of filters then $\S_{0}\xi=\operatorname{A}_{\mathbb{J}}\xi$
if and only if every singleton is $\xi$-cover- $\nicefrac{\mathbb{J_{*}}}{\mathbb{F}_{0*}}$-compact.
\end{cor}

In particular, $\S\xi=\S_{0}\xi$ (\footnote{a property we may call \emph{almost pretopological }by analogy with
the established notion \emph{almost topological }for $\S\xi=\T\xi$}) if and only if every singleton is cover-compact and $\S_{1}\xi=\S_{0}\xi$
if and only if every singleton is cover countably compact.
\begin{rem}
\label{rem:onedirectioncompactpoints} One direction of Corollary
\ref{cor:compactpoints} is true in general, which was already observed
in \cite{D.comp}: If $\mathbb{D}\subset\mathbb{J}$ are two classes
of filters, then $\operatorname{A}_{\mathbb{J}}\xi=\AdhD\xi$ whenever
every singleton is $\xi$-cover-$\nicefrac{\mathbb{J_{*}}}{\mathbb{D}_{*}}$-compact
(\footnote{Indeed, the inequality $\operatorname{A}_{\mathbb{J}}\xi\geq\AdhD\xi$
is true in general because $\mathbb{D}\subset\mathbb{J}$. On the
other hand, if $x\in\lim_{\AdhD\xi}\F$ and $\G\in\mathbb{J}$ with
$\G\#\F$ then $x\in\adh_{\xi}\G$ and thus $x\in\bigcap_{\mathbb{J}\ni\G\#\F}\adh_{\xi}\G=\lm_{\operatorname{A}_{\mathbb{J}}\xi}\F.$
Otherwise $x\notin\adh_{\xi}\G$ so that $\G_{c}\in\mathbb{J}_{*}$
is a cover of $\{x\}$ which has then a subcover in $\mathbb{D}_{*}$,
because $\{x\}$ is $\xi$-cover-$\dfrac{\mathbb{J}_{*}}{\mathbb{D}_{*}}$-compact.
In other words, there is $\D\in\mathbb{D}$ with $\D\leq\G$ and $x\notin\adh_{\xi}\D$,
which is not possible because $\D\#\F$ and $x\in\lim_{\AdhD\xi}\F$. }). On the other hand, the other implication is not true if $\mathbb{D}\neq\mathbb{F}_{0}$
(See Example \ref{exa:nonlindelofsingleton}).
\end{rem}

\section{Countable depth}

Paratopologies and hypotopologies form natural generalizations of
pretopologies, involving a countability condition. We now examine
others and how they relate.

Recall that $\xi$ is a pretopology if and only if for each $x\in|\xi|,$
the vicinity filter
\[
\V_{\xi}(x):=\bigwedge\left\{ \F:x\in\lm_{\xi}\F\right\} 
\]
converges to $x$. Thus pretopologies are exactly those convergences
admitting at each point a smallest convergent filter (i.e., the vicinity
filter). Therefore pretopologies can also be characterized as those
convergences $\xi$ satisfying
\begin{equation}
\lm_{\xi}\bigwedge_{\F\in\mathbb{A}}\F=\bigcap_{\F\in\mathbb{A}}\lm_{\xi}\F\label{eq:depth}
\end{equation}
 for every subset $\mathbb{A}$ of $\mathbb{F}|\xi|.$ 

We say that a convergence $\xi$ is \emph{countably deep }or \emph{has
countable depth }if (\ref{eq:depth}) holds for every \emph{countable
}subset $\mathbb{A}$ of $\mathbb{F}|\xi|$, and \emph{countably deep
for ultrafilters }if (\ref{eq:depth}) holds for every \emph{countable
}subset $\mathbb{A}$ of $\mathbb{U}|\xi|$.

We prove below (Corollary \ref{cor:cdeeppara}) that convergences
that are both countably deep for ultrafilters and paratopologies are
exactly pretopologies, so that any non-pretopological paratopology
is not countably deep, even for ultrafilters. In view of Corollary
\ref{cor:cdeeppara}, to produce a countably deep convergence that
is not a paratopology, it is enough to produce one that is not a pretopology,
which is easily done (e.g., Example \ref{exa:cdeepnotpara}). We will
also give an example of a pseudotopology that is countably deep for
ultrafilters but not countably deep (Example \ref{exa:Udeepbutnotdeep}). 

The classes of countably deep convergences and of countably deep for
ultrafilters pseudotopologies turn out to be reflective. Namely, given
a convergence $\xi$ let $\De\xi$ be defined by $x\in\lm_{\De\xi}\F$
if there is a countable subset $\mathbb{D}$ of $\lim_{\xi}^{-1}(x)$
with 
\[
\F\geq\bigwedge_{\D\in\mathbb{D}}\D,
\]
and let $\De^{\mathbb{U}}\xi$ be defined by $x\in\lm_{\De^{\mathbb{U}}\xi}\F$
if there is a countable subset $\mathbb{D}$ of $\lim_{\xi}^{-1}(x)\cap\mathbb{U}X$
with 
\[
\F\geq\bigwedge_{\U\in\mathbb{D}}\U.
\]

\begin{rem*}
We could more generally define $\operatorname{P}_{\kappa}\xi$ and
$\operatorname{P}_{\kappa}^{\mathbb{U}}\xi$ in a similar fashion
if $\mathbb{D}$ is restricted to sets of cardinality less than $\aleph_{\kappa}$.
In those terms, $\operatorname{P}_{0}$ is the reflector on convergences
of finite depth in the sense of \cite{DM.book}, where the notation
$\operatorname{L}$ was used instead. Pretopologies are the convergences
$\xi$ such that $\xi=\operatorname{P}_{\kappa}\xi$ for every $\kappa$.
\end{rem*}
\begin{prop}
The modifier
\begin{enumerate}
\item $\De$ is a concrete reflector and $\fix\De$ is the class of countably
deep convergences;
\item $\De^{\mathbb{U}}$ is an idempotent concrete functor;
\item $\S\De^{\mathbb{U}}$ is a concrete reflector and $\fix\S\De^{\mathbb{U}}$
is the class of pseudotopologies that are countably deep for ultrafilters.
\end{enumerate}
\end{prop}

\begin{proof}
It is plain that $\De$ is isotone, contractive, and idempotent and
that $\De^{\mathbb{U}}$ is isotone and idempotent. Moreover, 
\[
\xi\geq\S\xi\geq\S\De^{\mathbb{U}}\xi,
\]
for $\lim_{\xi}\U\subset\lim_{\De^{\mathbb{U}}\xi}\U$ if $\U\in\mathbb{U}X$.
Therefore, $\S\De^{\mathbb{U}}$ is isotone, contractive and idempotent.

If $f:|\xi|\to|\tau|$ is continuous and $x\in\lm_{\De\xi}\F$, that
is, there is a countable subset $\mathbb{D}$ of $\lim_{\xi}^{-1}(x)$
with $\F\geq\bigwedge_{\D\in\mathbb{D}}\D,$ then 
\[
f[\F]\geq f\left[\bigwedge_{\D\in\mathbb{D}}\D\right]\geq\bigwedge_{\D\in\mathbb{D}}f[\D]
\]
so that $f(x)\in\lm_{\De\tau}f[\F]$, for $f(x)\in\lim_{\tau}f[\D]$
for every $\D\in\mathbb{D}$, by continuity of $f$. A similar argument
applies to the effect that $f:|\De^{\mathbb{U}}\xi|\to|\De^{\mathbb{U}}\tau|$
is continuous. Therefore $\De$ and $\De^{\mathbb{U}}$ are functors,
and thus $\De$ and $\S\De^{\mathbb{U}}$ are concrete reflectors.

Note that a convergence is countably deep if and only if $\xi\leq\De\xi$
and thus $\fix\De$ is the class of countably deep convergences. 

On the other hand, $\S\De^{\mathbb{U}}\xi$ is a pseudotopology, and
is also countably deep for ultrafilters. Indeed, if 
\[
x\in\bigcap_{i\in\omega}\lm_{\S\De^{\mathbb{U}}\xi}\U_{i}=\bigcap_{i\in\omega}\lm_{\De^{\mathbb{U}}\xi}\U_{i}
\]
there is for each $i$ a sequence $(\W_{i,j})_{j\in\omega}$ of $\lim_{\xi}^{-1}(x)\cap\mathbb{U}X$
with $\U_{i}\geq\bigwedge_{j\in\omega}\W_{i,j}$ so that $\bigwedge_{i\in\omega}\U_{i}\geq\bigwedge_{(i,j)\in\omega^{2}}\W_{i,j}$
with $\{\W_{i,j}:(i,j)\in\omega^{2}\}\subset\lim_{\xi}^{-1}(x)\cap\mathbb{U}X$. 

On other hand, if $\xi$ is a countably deep for ultrafilters pseudotopology
and $x\in\lm_{\S\De^{\mathbb{U}}\xi}\F$ there for every $\U\in\beta\F$
there is a countable set $\mathbb{D}_{\U}\subset\lim_{\xi}^{-}(x)\cap\mathbb{U}X$
with $\U\geq\bigwedge_{\D\in\mathbb{D_{\mathcal{U}}}}\D$, so that
$x\in\lim_{\xi}\U$ because $\xi$ is countably deep for ultrafilters.
Thus $x\in\lim_{\xi}\F$ because $\xi=\S\xi$. Thus $\S\De^{\mathbb{U}}\xi=\xi$.
In other words, $\fix\S\De^{\mathbb{U}}$ is the class of pseudotopologies
that are countably deep for ultrafilters.
\end{proof}
Note that by definition,
\[
\De^{\mathbb{U}}\geq\De,
\]
and thus $\S\De^{\mathbb{U}}\geq\S\De$, but while $\S\De^{\mathbb{U}}\xi$
is always countably deep for ultrafilters, $\S\De\xi$ may fail to
be countably deep. Moreover, $\De\S\xi$ may fail to be a pseudotopology.
See the forthcoming paper \cite{DM.depth} for examples and details.
However, the class of countably deep pseudotopologies is reflective:
\begin{prop}
The class of countably deep pseudotopologies is reflective, as an
intersection of two reflective classes. Let $\De\triangle\S$ denote
the corresponding reflector. Then for every $\xi$ there are ordinals
$\alpha$ and $\beta$ with (\footnote{Here $F$ is one of the contractive modifiers $\S\De$ and $\De\S$
and 
\[
F^{1}=F,\;F^{\alpha}=F\left(\bigwedge_{\beta<\alpha}F^{\beta}\right).
\]
})
\[
\De\triangle\S\xi=(\S\De)^{\alpha}\xi=(\De\S)^{\beta}\xi.
\]
\end{prop}

\begin{proof}
This is an easy consequence of \cite[Proposition XIV.1.18]{DM.book}.
\end{proof}
\begin{thm}
\label{prop:S1D1} For every convergence $\xi$,
\[
\S_{1}\De^{\mathbb{U}}\xi=\S_{0}\xi.
\]
\end{thm}

\begin{proof}
Of course, $\S_{1}\xi\geq\S_{0}\xi$ and $\De^{\mathbb{U}}\xi\geq\S_{0}\xi$
so that 
\[
\S_{1}\De^{\mathbb{U}}\xi\geq\S_{1}\S_{0}\xi=\S_{0}\xi.
\]
Conversely, if $x\in\lm_{\S_{0}\xi}\F$ and $\H\in\mathbb{F}_{1}$
with decreasing filter-base $(H_{n})_{n\in\omega}$ with $\H\#\F$,
then for every $n\in\omega$, 
\[
x\in\adh_{\S_{0}\xi}H_{n}=\adh_{\xi}H_{n},
\]
so that there is $\U_{n}\in\beta(H_{n})$ with $x\in\lm_{\xi}\U_{n}$.
Then 
\[
x\in\lm{}_{\De^{\mathbb{U}}\xi}\left(\bigwedge_{n\in\omega}\U_{n}\right)
\]
 and $\bigwedge_{n\in\omega}\U_{n}\#\H$ so that $x\in\adh_{\De^{\mathbb{U}}\xi}\H$
and thus $x\in\lm_{\S_{1}\De^{\mathbb{U}}\xi}\F$.
\end{proof}
In view of Corollary \ref{cor:compactpoints}, we conclude:
\begin{cor}
\label{cor:ccdeeptocomppt} If a convergence $\xi$ is countably deep
for ultrafilters then every singleton is cover-$\nicefrac{\mathbb{F}_{1*}}{\mathbb{F}_{0*}}$-compact.
\end{cor}

In restriction to paratopologies, we obtain:

\begin{cor}
\label{cor:cdeeppara} The following are equivalent for a paratopology
$\xi$:
\begin{enumerate}
\item $\xi$ is countably deep;
\item $\xi$ is countably deep for ultrafilters;
\item $\xi$ is a pretopology;
\item every singleton is cover-$\nicefrac{\mathbb{F}_{1*}}{\mathbb{F}_{0*}}$-compact.
\end{enumerate}
\end{cor}

\begin{proof}
$(1)\then(2)$ by definition, $(2)\then(3)$ because if $\xi=\S_{1}\xi=\De^{\mathbb{U}}\xi$
then $\xi=\S_{0}\xi$ by Theorem \ref{prop:S1D1}, $(3)\iff(4)$ is
an instance of Corollary \ref{cor:compactpoints}, and $(3)\then(1)$
is clear.
\end{proof}
We will see (Example \ref{exa:S1S0notcdeep}) that the converse of
Corollary \ref{cor:ccdeeptocomppt} is false, that is, there are convergences
$\xi$ satisfying $\S_{1}\xi=\S_{0}\xi$ that are not countably deep
for ultrafilters.
\begin{example}[A pseudotopology of countable depth that is not a paratopology]
\label{exa:cdeepnotpara} Let $X$ be an uncountable set and define
on $X$ the prime convergence $\xi$ in which the only non-isolated
point $x_{0}$ satisfies 
\[
x_{0}\in\lm_{\xi}\F\iff\ker\F\subset\{x_{0}\}\text{ and }[X]^{\leq\omega}\cap\F\neq\emptyset.
\]

This convergence has countable depth. Indeed, if $x_{0}\in\lim_{\xi}\F_{n}$
for $n\in\omega$, then for each $n\in\omega$ there is $C_{n}\in[X]^{\leq\omega}\cap\F_{n}$
and thus $\bigcup_{n\in\omega}C_{n}\in[X]^{\leq\omega}\in\bigwedge_{n\in\omega}\F_{n}$.
Moreover, 
\[
\ker\bigwedge_{n\in\omega}\F_{n}\subset\bigcup_{n\in\mathbb{N}}\ker\F_{n}\subset\{x_{0}\},
\]
by Lemma \ref{lem:kernels}. Thus $x_{0}\in\lim_{\xi}\bigwedge_{n\in\omega}\F_{n}$.

Moreover, $\xi=\S\xi$. Indeed, if $\F$ is such that $x_{0}\in\lim_{\xi}\U$
for every $\U\in\beta\F$, then for every $\U\in\beta\F$ there us
$C_{\U}\in\U\cap[X]^{\leq\omega}$ and $\ker\U\subset\{x_{0}\}$.
Thus $\ker\F=\ker\bigwedge_{\U\in\beta\F}\U\subset\{x_{0}\}$ and
there is a finite set $\mathbb{D}$ of $\beta\F$ with $\bigcup_{\U\in\mathbb{D}}C_{\U}\in\F$.
Clearly, $\bigcup_{\U\in\mathbb{D}}C_{\U}\in[X]^{\leq\omega}$. Therefore
$x_{0}\in\lim_{\xi}\F$.

Finally, $\xi$ is not pretopological, so that, in view of Corollary
\ref{cor:cdeeppara}, $\xi$ is not paratopological. To see this,
note that 
\[
\V_{\xi}(x_{0})=\bigwedge\left\{ \U\in\mathbb{U}X:x_{0}\in\lm_{\xi}\U\right\} 
\]
is coarser than the cofinite filter on $X$, which does not have any
countable element, hence does not converge to $x_{0}$ for $\xi$.
Indeed, if $A\in\V_{\xi}(x_{0})$ then, as for every $C\in[X]^{\omega}$
there is $\U\in\beta C$ with $x_{0}\in\lim_{\xi}\U$, we have $A\cap C\neq\emptyset$.
Hence $A$ can only miss a finite set, that is, $A$ is cofinite. 
\end{example}

On the other hand, $\De$ and $\S_{1}$ do not commute, as we may
have 
\[
\De\S_{1}\xi\geq\S\De\S_{1}\xi>\S_{0}\xi,
\]
while $\S_{1}\De\xi=\S_{0}\xi$ (Example \ref{exa:D1S1donotcommute}).
To produce this example, let us develop a little bit more machinery.

\section{\label{sec:inStone}Representations of convergence notions in the
Stone topology}

We will characterize various properties of a convergence $\xi$ in
terms of the topological features of sets of the form 
\[
\mathbb{U}_{\xi}(x):=\left\{ \U\in\mathbb{U}X:\;x\in\lm_{\xi}\U\right\} 
\]
 in the Stone space $\mathbb{U}X$. We restrict ourselves to pseudotopologies
in this section, equivalently, we assume that 
\[
\beta\F\subset\mathbb{U}_{\xi}(x)\then x\in\lm_{\xi}\F,
\]
for all $x\in|\xi|$. 

Conversely, a family $\B=\{B_{x}:x\in X\}$ of subsets of $\mathbb{U}X$
satisfying $\{x\}^{\uparrow}\in B_{x}$ for each $x\in X$ determines
a unique pseudotopology $\xi_{\B}$ defined (\footnote{Note that then
\[
\F\leq\G\then\lim\F\subset\lim\G
\]
rephrases as 
\[
\beta\G\subset\beta\F\text{ and }\beta\F\subset B_{x}\then\beta\G\subset B_{x}.
\]
}) by 
\[
x\in\lm_{\xi_{\B}}\F\iff\beta\F\subset B_{x}.
\]

\subsection{Closures on $\mathbb{U}X$ and their coincidence on $\mathbb{U}_{\xi}(x)$\label{subsec:Closureson Uxix}}

With a class $\mathbb{D}$ of filters, a closure operator $\cl_{\mathbb{D}^{*}}$
on $\mathbb{U}X$ was associated in \cite{mynard2007closure} by declaring
$\{\beta\D:\D\in\mathbb{D}\}$ a base of open sets, that is,
\[
\cl_{\mathbb{D}^{*}}A=\bigcap\left\{ \mathbb{U}X\setminus\beta\D:\D\in\mathbb{D},A\subset\mathbb{U}X\setminus\beta\D\right\} .
\]

\cite[Prop. 4]{mynard2007closure} clarifies when this closure operator
is a topological closure. It is in particular the case for the classes
$\mathbb{F}_{0}$ (yielding the usual topology of $\mathbb{U}X$:
$\cl_{\mathbb{F}_{0}^{*}}=\cl_{\beta}$), $\mathbb{F}_{1}$ (yielding
the $G_{\delta}$-topology: $\cl_{\mathbb{F}_{1}^{*}}=\cl_{\delta}$),
$\mathbb{F}_{\wedge1}$ (for which the topology is harder to describe),
and $\mathbb{F}$ (yielding the discrete topology). It turns out that:
\begin{thm}
\label{thm:ADclosureinUX} Let $\xi$ be a pseudotopology. Then
\[
\mathbb{U}_{\AdhD\xi}(x)=\cl_{\mathbb{D}^{*}}\left(\mathbb{U}_{\xi}(x)\right).
\]
\end{thm}

\begin{proof}
\[
\U\in\mathbb{U}_{\AdhD\xi}(x)\iff\forall\D\in\mathbb{D},\D\subset\U\then\beta\D\cap\mathbb{U}_{\xi}(x)\neq\emptyset.
\]
Since the sets of the form $\beta\D$ for $\D\in\mathbb{D}$ and $\D\subset\U$
form a basis of neighborhood of $\U$ for the closure $\mathbb{D}^{*}$,
the conclusion follows.\medskip{}
\end{proof}
\begin{center}
\begin{tabular}{|c|c|c|}
\hline 
$\mathbb{D}$ & $\mathbb{D}^{*}$ & $\mathbb{U}_{\operatorname{A}_{\mathbb{D}}\xi}(x)$\tabularnewline
\hline 
\hline 
$\mathbb{F}_{0}$ & $\beta$ & $\mathbb{U}_{\S_{0}\xi}(x)=\cl_{\beta}\left(\mathbb{U}_{\xi}(x)\right)$\tabularnewline
\hline 
$\mathbb{F}_{1}$ & $\delta$-topology of $\beta$ & $\mathbb{U}_{\S_{1}\xi}(x)=\cl_{\delta}\left(\mathbb{U}_{\xi}(x)\right)$\tabularnewline
\hline 
$\mathbb{F}_{\wedge1}$ & generated by $\beta\F$ with $\N_{\beta}(\beta\F)=\N_{\delta}(\beta\F)$ & $\mathbb{U}_{\S_{\wedge1}\xi}(x)=\cl_{(\mathbb{F}_{\wedge1})^{*}}\left(\mathbb{U}_{\xi}(x)\right)$\tabularnewline
\hline 
$\mathbb{F}$ & discrete & $\mathbb{U}_{\S\xi}(x)=\mathbb{U}_{\xi}(x)$\tabularnewline
\hline 
\end{tabular}
\par\end{center}

\medskip{}
In particular, $\xi$ is a pretopology if and only if for every $x\in|\xi|$,
the set $\mathbb{U}_{\xi}(x)$ is closed; a paratopology if and only
if for every $x\in|\xi|$, the set $\mathbb{U}_{\xi}(x)$ is $\delta$-closed
(that is, closed for the $G_{\delta}$ topology of $\beta$).
\begin{cor}
\label{cor:equalclosures} Let $\mathbb{D}$ and $\mathbb{J}$ be
two classes of filters. A convergence $\xi$ satisfies 
\[
\AdhD\xi=\operatorname{A}_{\mathbb{J}}\xi
\]
if and only if
\[
\cl_{\mathbb{D}^{*}}\left(\mathbb{U}_{\xi}(x)\right)=\cl_{\mathbb{J}^{*}}\left(\mathbb{U}_{\xi}(x)\right)
\]
 for every $x\in|\xi|$. 
\end{cor}

\begin{prop}
\label{prop:Uxcompactnessgeneral} Let $\xi$ be a convergence, $x\in|\xi|$
and let $\kappa<\lambda$ be two cardinals. Then $\mathbb{U}_{\xi}(x)$
is cover-$\nicefrac{\mathbb{F}_{\lambda*}}{\mathbb{F}_{\kappa_{*}}}$-compact
if and only if for every family $\mathbb{D}$ of filters of cardinality
less than $\lambda$
\begin{equation}
x\notin\adh_{\xi}\bigvee_{\D\in\mathbb{D}}\D\then\exists S\subset\mathbb{D},|S|<\kappa:x\notin\adh_{\xi}\bigvee_{\D\in S}\D.\label{eq:UxF1comp}
\end{equation}
\end{prop}

\begin{proof}
Assume $\mathbb{U}_{\xi}(x)$ is cover-$\nicefrac{\mathbb{F}_{\lambda*}}{\mathbb{F}_{\kappa_{*}}}$-compact
and assume that $\mathbb{D}$ is a collection of filters of cardinality
less than $\lambda$ with $x\notin\adh_{\xi}\bigvee_{\D\in\mathbb{D}}\D$
equivalently,
\[
\beta(\bigvee_{\D\in\mathbb{D}}\D)=\bigcap_{\D\in\mathbb{D}}\beta\D\subset\mathbb{U}X\setminus\mathbb{U}_{\xi}(x),
\]
that is, 
\[
\mathbb{U}_{\xi}(x)\subset\bigcup_{\D\in\mathbb{D}}\mathbb{U}X\setminus\beta\D.
\]
By $\nicefrac{\mathbb{F}_{\lambda*}}{\mathbb{F}_{\kappa_{*}}}$-compactness,
there is a subset $S$ of $\mathbb{D}$ of cardinality less than $\kappa$
with 
\[
\mathbb{U}_{\xi}(x)\subset\bigcup_{\D\in S}\mathbb{U}X\setminus\beta\D\iff\beta(\bigvee_{\D\in S}\D)\subset\mathbb{U}X\setminus\mathbb{U}_{\xi}(x),
\]
 that is, $x\notin\adh_{\xi}\bigvee_{\D\in S}\D$.

Conversely, assume (\ref{eq:UxF1comp}) and let $\C$ be an open cover
of $\mathbb{U}_{\xi}(x)$ with $\C_{c}\in\mathbb{F}_{\lambda}$. For
every $C$ in an ideal base $\B$ of cardinality less than $\lambda$,
there is $\D_{C}\in\mathbb{F}X$ with $C=\mathbb{U}X\setminus\beta\D_{C}$.
Let $\mathbb{D}=\{\D_{C}:C\in\B\}$. Then 
\begin{eqnarray*}
\mathbb{U}_{\xi}(x)\subset\bigcup_{\D\in\mathbb{D}}\mathbb{U}X\setminus\beta\D & \iff & \beta(\bigvee_{\D\in\mathbb{D}}\D)=\bigcap_{\D\in\mathbb{D}}\beta\D\subset\mathbb{U}X\setminus\mathbb{U}_{\xi}(x)\\
 &  & x\notin\adh_{\xi}\bigvee_{\D\in\mathbb{D}}\D,
\end{eqnarray*}
so that, in view of (\ref{eq:UxF1comp}), there is a subset $S$ of
$\mathbb{D}$ of cardinality less than $\kappa$ with $x\notin\adh_{\xi}\bigvee_{\D\in S}\D$,
equivalently, 
\[
\mathbb{U}_{\xi}(x)\subset\bigcup_{\D\in S}\mathbb{U}X\setminus\beta\D,
\]
so that $\{\mathbb{U}X\setminus\beta\D\}$ is a subcover of $\C$
of cardinality less than $\kappa$.
\end{proof}
We now relate various kinds of compactness of $\mathbb{U}_{\xi}(x)$
to the corresponding kind of (cover) compactness of $\{x\}$.

We say that a class $\mathbb{D}_{*}$ of ideals is $\beta$-\emph{compatible
}if 
\begin{equation}
\C\in\mathbb{D}_{*}(X)\then\left\{ \beta C:C\in\C\right\} \in\mathbb{D}_{*}(\mathbb{U}X),\label{eq:XtoUX}
\end{equation}
and 
\begin{equation}
\{\beta(X\setminus D):D\in\D\}\in\mathbb{D}_{*}(\mathbb{U}X)\then\D\in\mathbb{D}(X).\label{eq:UXtoX}
\end{equation}

\begin{lem}
For every $\kappa$, $\mathbb{F}_{\kappa*}$ and $\mathbb{F}_{*}$
are $\beta$-compatible.
\end{lem}

\begin{prop}
\label{prop:compactnessUx} Let $\xi$ be a convergence and $x\in|\xi|$
and let $\mathbb{D\subset J}$ be two classes of filters, where $\mathbb{J}$
is $\beta$-compatible. If $\mathbb{U}_{\xi}(x)$ is cover-$\nicefrac{\mathbb{J_{*}}}{\mathbb{D}_{*}}$-compact
then $\{x\}$ is cover-$\nicefrac{\mathbb{J_{*}}}{\mathbb{D}_{*}}$-compact.
Moreover, if $\mathbb{J}=\mathbb{F}$ and $\mathbb{D}$ is $\beta$-compatible,
the converse is true. 
\end{prop}

Note that 
\begin{equation}
x\notin\adh_{\xi}\D\iff\mathbb{U}_{\xi}(x)\subset\left(\beta\D\right)^{c}=\bigcup_{D\in\D}\beta(X\setminus D).\label{eq:coverpoint}
\end{equation}

\begin{proof}
Assume that $\G\in\mathbb{J}$ with $x\notin\adh_{\xi}\G$. By (\ref{eq:coverpoint}),
$\{\beta(X\setminus G):G\in\G\}$ is an open cover of $\mathbb{U}_{\xi}(x)$,
which belongs to $\mathbb{J}_{*}$ by $\beta$-compatibility. By (2),
it has a subcover in $\mathbb{D}_{*}$, that is, in view of (\ref{eq:UXtoX}),
there is $\D\in\mathbb{D}$, $\D\leq\G$ with $\mathbb{U}_{\xi}(x)\subset\bigcup_{D\in\D}\beta(X\setminus D)$.
In view of (\ref{eq:coverpoint}), $x\notin\adh_{\xi}\D$. Hence $\{x\}$
is cover-$\nicefrac{\mathbb{J_{*}}}{\mathbb{D}_{*}}$-compact. 

Assume now that $\{x\}$ is cover-$\nicefrac{\mathbb{F_{*}}}{\mathbb{D}_{*}}$-compact
and let $\C$ be an open cover of $\mathbb{U}_{\xi}(x)$. It has a
refinement of the form $\left\{ \beta A:A\in\A\right\} $ for some
family $\A\subset\mathbb{P}X$, which covers $\mathbb{U}_{\xi}(x)$.
As $\beta(A\cup B)=\beta A\cup\beta B$, we can assume $\A$ and $\{\beta A:A\in\A\}$
to be ideals. Hence $\F=\A_{c}$ is a filter on $X$ with $x\notin\adh_{\xi}\F$
because of (\ref{eq:coverpoint}). Moreover, $\{x\}$ is cover-$\nicefrac{\mathbb{F_{*}}}{\mathbb{D}_{*}}$-compact
so that there is $\D\leq\F$ with $x\notin\adh_{\xi}\D$, equivalently,
$\{\beta(X\setminus D):D\in\D\}\subset\{\beta A:A\in\A\}$ is a cover
of $\mathbb{U}_{\xi}(x)$, which is in $\mathbb{D}_{*}$ by $\beta$-compatibility
of $\mathbb{D}$. Hence $\C$ has a refinement in $\mathbb{D}_{*}$
that covers $\mathbb{U_{\xi}}(x)$, which is thus cover-$\nicefrac{\mathbb{F_{*}}}{\mathbb{D}_{*}}$-compact.
\end{proof}
In the case where $\mathbb{D}=\mathbb{F}_{0}$ and $\mathbb{J}=\mathbb{F}$,
Proposition \ref{prop:compactnessUx}, Proposition \ref{prop:Uxcompactnessgeneral}
and Corollary \ref{cor:equalclosures} may be coupled with Corollary
\ref{cor:compactpoints} to the effect that:
\begin{cor}
\label{cor:almostpretop} The following are equivalent for a convergence
$\xi$:
\begin{enumerate}
\item $\S\xi=\S_{0}\xi$;
\item $\cl_{\beta}\left(\mathbb{U}_{\xi}(x)\right)=\mathbb{U}_{\xi}(x)$
for every $x\in|\xi|$;
\item $\{x\}$ is cover-compact for every $x\in|\xi|$;
\item $\mathbb{U}_{\xi}(x)$ is compact for every $x\in|\xi|$;
\item 
\[
x\notin\adh_{\xi}\bigvee_{\D\in\mathbb{D}}\D\then\exists S\in[\mathbb{D}]^{<\infty}:x\notin\adh_{\xi}\bigvee_{\D\in S}\D.
\]
\end{enumerate}
\end{cor}

In case $\mathbb{J=F}$ and $\mathbb{D}=\mathbb{F}_{1}$, we obtain
the following from Proposition \ref{prop:compactnessUx} and Proposition
\ref{prop:Uxcompactnessgeneral}:
\begin{cor}
\label{cor:Lindelofptwise} Let $\xi$ be a convergence and $x\in|\xi|$.
The following are equivalent:
\begin{enumerate}
\item $\{x\}$ is cover-Lindelöf;
\item $\mathbb{U}_{\xi}(x)$ is Lindelöf;
\item 
\[
x\notin\adh_{\xi}\bigvee_{\D\in\mathbb{D}}\D\then\exists S\in[\mathbb{D}]^{\omega}:x\notin\adh_{\xi}\bigvee_{\D\in S}\D.
\]
\end{enumerate}
\end{cor}

For the same classes, Proposition \ref{prop:compactnessUx} combined
with Corollary \ref{cor:equalclosures} yields:
\begin{cor}
\label{cor:SisS1} Let $\xi$ be a convergence. Then $\S\xi=\S_{1}\xi$
if and only if $\mathbb{U}_{\xi}(x)$ is $\delta$-closed for every
$x\in|\xi|$.
\end{cor}

To summarize
\[
\xymatrix{\forall x\;\mathbb{U}_{\xi}(x)\text{ is cover-}\nicefrac{\mathbb{J_{*}}}{\mathbb{D}_{*}}\text{-compact}\ar@{=>}[d] & \forall x\;\cl_{\mathbb{D}^{*}}\left(\mathbb{U}_{\xi}(x)\right)=\cl_{\mathbb{J}^{*}}\left(\mathbb{U}_{\xi}(x)\right)\ar@{=>}[d]\\
\forall x\;\{x\}\text{ is cover-}\nicefrac{\mathbb{J_{*}}}{\mathbb{D}_{*}}\text{-compact}\ar@{=>}[r]\ar@{:>}@(ul,dl)[u]^{\mathbb{J}=\mathbb{F}} & \AdhD\xi=\operatorname{A}_{\mathbb{J}}\xi\ar@{:>}|\times@(dl,dr)[l]_{\text{Example \ref{exa:nonlindelofsingleton}}}\ar@{=>}[u]\ar@{:>}@(ul,ur)[l]^{\mathbb{D}=\mathbb{F}_{0}}
}
\]

I do not know at the moment if the converse of Proposition \ref{prop:compactnessUx}
may fail when $\mathbb{J}\neq\mathbb{F}$. For instance,
\begin{problem}
Is there a convergence $\xi$ with $x\in|\xi|$ such that $\{x\}$
cover-countably compact but $\mathbb{U_{\xi}}(x)$ is not countably
compact?
\end{problem}

Note also that:
\begin{prop}
Let $\xi$ be a pseudotopology. Then 
\[
\mathbb{U}_{\I_{1}\xi}(x)=\mathbb{U}_{\S\I_{1}\xi}(x)=\intr_{\delta}\left(\mathbb{U}_{\xi}(x)\right).
\]
\end{prop}

\begin{cor}
A convergence is bisequential (respectively strongly Fréchet, respectively
Fréchet) if and only if $\mathbb{U}_{\xi}(x)$ is $\delta$-open (respectively,
$\mathbb{U}_{\xi}(x)=\cl_{\delta}\intr_{\delta}\mathbb{U}_{\xi}(x)$,
respectively $\mathbb{U}_{\xi}(x)=\cl_{\beta}\intr_{\delta}\mathbb{U}_{\xi}(x)$)
for every $x\in|\xi|$.
\end{cor}

The interested reader may consult \cite{mynard2007closure} for further
details on such characterizations of the classes of bisequential,
strongly Fréchet, Fréchet spaces and other similar classes of spaces.

\subsection{Countable depth vs countable depth for ultrafilters}

Let $\mathcal{K}_{\sigma}(X)$ be the set of $\sigma$-compact subsets
of $X$. Note that as compactness is absolute, $\sigma$-compact subsets
of $A\subset\mathbb{U}X$ have the form $\bigcup_{i\in\omega}\beta\F_{i}$
where $\F_{i}\in\mathbb{F}X$ for every $i\in\omega$.

Let $\cl_{t}$ denote the closure in the countably tight modification
of $\mathbb{U}X$, that is, given $A\subset\mathbb{U}X$,
\[
\cl_{t}A=\bigcup_{C\in\left[A\right]^{\leq\omega}}\cl_{\beta}C.
\]

\begin{prop}
\label{prop:D1inUX} Let $\xi$ be a pseudotopology. Then
\begin{equation}
\mathbb{U}_{\De\xi}(x)=\bigcup_{F\in\mathcal{K}_{\sigma}(\mathbb{U}_{\xi}(x))}\cl_{\beta}F,\label{eq:UD1}
\end{equation}
and 
\begin{equation}
\mathbb{U}_{\De^{\mathbb{U}}\xi}(x)=\cl_{t}\mathbb{U}_{\xi}(x).\label{eq:UDU1}
\end{equation}

In particular, $\xi$ has countable depth if and only if for every
$x\in|\xi|$,
\[
\cl_{\beta}F\subset\mathbb{U}_{\xi}(x)
\]
whenever $F$ is a $\sigma$-compact subset of $\mathbb{U}_{\xi}(x)$.
\end{prop}

\begin{proof}
As $\xi$ is a pseudotopology, $\U\in\mathbb{U}_{\De\xi}(x)$ if and
only if $\U\in\beta(\bigwedge_{n\in\omega}\H_{n})$ for a sequence
$(\H_{n})_{n\in\omega}$ of filters with $\beta\H_{n}\subset\mathbb{U}_{\xi}(x)$.
The result follows from the observation that sets of the form (\ref{eq:betainf})
are exactly closures of $\sigma$-compact subsets of $\mathbb{U}_{\xi}(x)$.

As for (\ref{eq:UDU1}), note that $\U\in\mathbb{U}_{\De^{\mathbb{U}}\xi}(x)$
if and only if there is a sequence $\{\W_{n}:n\in\omega\}$ on $\mathbb{U}_{\xi}(x)$
with $\U\geq\bigwedge_{n\in\omega}\W_{n}$, so that $\U\in\beta\left(\bigwedge_{n\in\omega}\W_{n}\right)=\cl_{\beta}\{\W_{n}:n\in\omega\}$.
In other words, $\U\in\mathbb{U}_{\De^{\mathbb{U}}\xi}(x)$ if and
only if $\U\in\cl_{t}\mathbb{U}_{\xi}(x)$. 
\end{proof}
Therefore, Corollary \ref{cor:ccdeeptocomppt} can be refined:
\begin{cor}
If $\xi$ is countably deep for ultrafilters then $\mathbb{U}_{\xi}(x)$
is countably compact for every $x\in|\xi|$.
\end{cor}

\begin{proof}
Every countable subset of $\mathbb{U}_{\xi}(x)$ has an accumulation
point in $\mathbb{U}X$ by compactness. In view of Proposition \ref{prop:D1inUX},
$\mathbb{U}_{\xi}(x)=\cl_{t}(\mathbb{U}_{\xi}(x))$ and thus this
accumulation point is in $\mathbb{U}_{\xi}(x)$, which is thus countably
compact.
\end{proof}
\begin{example}[A pseudotopology that is countably deep for ultrafilters but not countably
deep]
\label{exa:Udeepbutnotdeep} Following an idea of \cite{nyikosOPIT2},
take a weak P-point $\U_{0}$ of $\mathbb{U}^{\circ}\mathbb{N}$ that
is not a P-point (in $\mathbb{U}^{\circ}\mathbb{N}$), and take a
prime pseudotopology $\xi$ on $\mathbb{N}$ with only non-isolated
point $1$, where
\[
\mathbb{U}_{\xi}(1)=\{\{1\}^{\uparrow}\}\cup\mathbb{U^{\circ}\mathbb{N}\setminus}\{\U_{0}\}.
\]
Then $\mathbb{U}_{\xi}(1)=\cl_{t}(\mathbb{U}_{\xi}(1))$ because $\U_{0}$
is a weak P-point, so that $\xi=\De^{\mathbb{U}}\xi$, but there is
an $F_{\sigma}$ subset $S$ of $\mathbb{U}^{\circ}\mathbb{N}$ (hence
of $\mathbb{U}\mathbb{N}$) with $\U_{0}\in\cl_{\beta}S$, because
$\U_{0}$ is not a P-point. Thus $\xi$ is not countably deep.
\end{example}

\begin{example}[A convergence $\xi$ with $\S_{1}\xi=\S_{0}\xi$ that is not countably
deep for ultrafilters]
\label{exa:S1S0notcdeep} Take similarly a free ultrafilter $\U_{0}$
of $\mathbb{U}^{\circ}\mathbb{N}$ that is not a weak P-point, and
take a prime pseudotopology $\xi$ on $\mathbb{N}$ with only non-isolated
point $1$, where
\[
\mathbb{U}_{\xi}(1)=\{\{1\}^{\uparrow}\}\cup\mathbb{U^{\circ}\mathbb{N}\setminus}\{\U_{0}\}.
\]
Then $\U_{0}\in\cl_{t}(\mathbb{U}_{\xi}(1))\setminus\mathbb{U}_{\xi}(1)$
$\xi\neq\De^{\mathbb{U}}\xi$, but 
\[
\cl_{\beta}(\mathbb{U}_{\xi}(1))=\cl_{\delta}(\mathbb{U}_{\xi}(1))=\{\{1\}^{\uparrow}\}\cup\mathbb{U^{\circ}\mathbb{N}},
\]
so that $\S_{1}\xi=\S_{0}\xi$.
\end{example}

Let us now return to showing that $\De$ and $\S_{1}$ do not commute.
Let's start by examining when $\De\xi$ and $\S_{0}\xi$ coincide:
\begin{prop}
\label{prop:D1isS0} The following are equivalent for a pseudotopology
$\xi$:
\begin{enumerate}
\item $\De\xi=\S_{0}\xi$;
\item For every $x\in|\xi|$, there is $\{\D_{n}:n\in\omega\}\subset\lim_{\xi}^{-}(x)$
with $\V_{\xi}(x)=\bigwedge_{n\in\omega}\D_{n}$;
\item there is a $\sigma$-compact subset $F$ of $\mathbb{U}_{\xi}(x)$
with $\cl_{\beta}F=\cl_{\beta}\mathbb{U}_{\xi}(x)$.
\end{enumerate}
\end{prop}

\begin{proof}
$(1)\then(2)$: If $\De\xi=\S_{0}\xi$ then $x\in\lm_{\De\xi}\V_{\xi}(x)$
for each $x$, that is, $(2)$. 

$(2)\then(3)$ because 
\begin{equation}
\cl_{\beta}\mathbb{U}_{\xi}(x)=\beta(\V_{\xi}(x))=\beta\Big(\bigwedge_{n\in\omega}\D_{n}\Big)=\cl_{\beta}\Big(\bigcup_{n\in\omega}\beta\D_{n}\Big).\label{eq:aux}
\end{equation}

$(3)\then(1)$ Since $F=\bigcup_{n\in\omega}\beta\D_{n}$ for a sequence
$\{\D_{n}:n\in\omega\}$ of filters with $\beta\D_{n}\subset\mathbb{U}_{\xi}(x)$
and $\xi$ is pseudotopological, we conclude that $\{\D_{n}:n\in\omega\}\subset\lim_{\xi}^{-}(x)$.
From (\ref{eq:aux}), we conclude that $\V_{\xi}(x)$ converges for
$\De\xi$, that is, $\De\xi=\S_{0}\xi$.
\end{proof}
\begin{thm}
\label{thm:deltaclosednotclosureofFsigma} There is a set $X$ and
a non-closed $\delta$-closed subset $S$ of $\mathbb{U}X$ such that
$\cl_{\beta}S\neq\cl_{\beta}C$ whenever $C$ is an $F_{\sigma}$
(in $\mathbb{U}X$) subset of $S$.
\end{thm}

Subsection \ref{subsec:proofofTheorem} below provides a proof of
Theorem \ref{thm:deltaclosednotclosureofFsigma}.
\begin{example}[$\S\De\S_{1}\xi>\S_{0}\xi$]
 \label{exa:D1S1donotcommute} Let $X$ and $S$ be as in Theorem
\ref{thm:deltaclosednotclosureofFsigma} and let $x_{0}\in X$. Let
$\xi$ be a prime pseudotopology on $X$ with non-isolated point $x_{0}$
and $\mathbb{U}_{\xi}(x_{0})=\{\{x_{0}\}^{\uparrow}\}\cup S$. Then
$\xi=\S_{1}\xi>\S_{0}\xi$ because $\mathbb{U}_{\xi}(x_{0})$ is $\delta$-closed
but not closed. On the other hand, in view of Proposition \ref{prop:D1isS0},
$\S\De\xi>\S_{0}\xi$ because $\cl_{\beta}S\neq\cl_{\beta}C$ whenever
$C$ is an $F_{\sigma}$ subset of $S$.
\end{example}

\subsection{\label{subsec:proofofTheorem}Topological properties of $\mathbb{U}_{\xi}X$}

We observe an interplay between coincidence of different closures
for $\mathbb{U}_{\xi}X$ and various kind of compactness of $\mathbb{U}_{\xi}X$
as we did for $\mathbb{U}_{\xi}(x)$ in Section \ref{subsec:Closureson Uxix}.
\begin{prop}
\label{prop:clUxix} $\U\in\cl_{\mathbb{D}^{*}}(\mathbb{U}_{\xi}X)$
if and only if $\U$ is $\mathbb{D}$-compactoid. 
\end{prop}

\begin{proof}
$\U\in\cl_{\mathbb{D}^{*}}(\mathbb{U}_{\xi}X)$ if for every $\D\in\mathbb{D}$
with $\D\leq\U$ we have $\beta\D\cap\mathbb{U}_{\xi}X\neq\emptyset$,
that is, 
\[
\mathbb{D}\ni\D\#\U\then\adh_{\xi}\D\neq\emptyset.
\]
In other words, $\U\in\cl_{\mathbb{D}^{*}}(\mathbb{U}_{\xi}X)$ if
and only if $\U$ is $\mathbb{D}$-compactoid.
\end{proof}
Recall that $\mathbb{K}_{\mathbb{D}}$ denote the classes of all $\mathbb{D}$-compactoid
filters (for a given convergence).
\begin{prop}
\label{prop:clDisClJUxiX} Let $\xi$ be a convergence and $\mathbb{D}\subset\mathbb{J}$
be two classes of filters. The following are equivalent:
\begin{enumerate}
\item $\cl_{\mathbb{D^{*}}}(\mathbb{U}_{\xi}X)=\cl_{\mathbb{J}^{*}}(\mathbb{U}_{\xi}X)$;
\item $\mathbb{K}_{\mathbb{D}}(\xi)=\mathbb{K}_{\mathbb{J}}(\xi)$, that
is, every $\mathbb{D}$-compactoid filter on $\xi$ is $\mathbb{J}$-compactoid.
\end{enumerate}
\end{prop}

\begin{proof}
The inclusions $\cl_{\mathbb{J}^{*}}\subset\cl_{\mathbb{D}^{*}}$
and $\mathbb{K}_{\mathbb{J}}(\xi)\subset\mathbb{K}_{\mathbb{D}}(\xi)$
are always true because $\mathbb{D}\subset\mathbb{J}$. In view of
Proposition \ref{prop:clUxix}, $\cl_{\mathbb{D}^{*}}(\mathbb{U}_{\xi}X)\subset\cl_{\mathbb{J}^{*}}(\mathbb{U}_{\xi}X)$
if and only if $\mathbb{K}_{\mathbb{D}}(\xi)\subset\mathbb{K}_{\mathbb{J}}(\xi)$.
\end{proof}
Note that $\mathbb{K}_{\mathbb{F}_{0}}=\mathbb{F}$ because every
filter is $\mathbb{F}_{0}$-compactoid. Hence if $\mathbb{D}=\mathbb{F}_{0}$
we obtain:
\begin{cor}
Let $\mathbb{J}$ be a class of filters and $\xi$ be a convergence.
Then
\[
\cl_{\beta}(\mathbb{U}_{\xi}X)=\cl_{\mathbb{J}}(\mathbb{U}_{\xi}X)
\]
if and only if $\xi$ is $\mathbb{J}$-compact.
\end{cor}

\begin{proof}
By Proposition \ref{prop:clDisClJUxiX}, 
\[
\cl_{\beta}(\mathbb{U}_{\xi}X)=\cl_{\mathbb{J}}(\mathbb{U}_{\xi}X)\iff\mathbb{K}_{\mathbb{F}_{0}}(\xi)=\mathbb{F}(\xi)=\mathbb{K}_{\mathbb{J}}(\xi)
\]
so that every filter is $\mathbb{J}$-compactoid. In particular, every
$\G\in\mathbb{J}$ is $\mathbb{J}$-compactoid and thus $\adh_{\xi}\G\neq\emptyset$.
Conversely, if $\xi$ is $\mathbb{J}$-compact then every filter in
$\mathbb{J}$ has non-empty adherence and thus every filter is $\mathbb{J}$-compactoid.
\end{proof}
Now let's consider the case where $\mathbb{J}=\mathbb{F}$ in Proposition
\ref{prop:clDisClJUxiX}. Then $\cl_{\mathbb{F^{*}}}$ is the identity
because $\mathbb{F}^{*}$ is the discrete topology. 

An ideal cover $\C$ of $(X,\xi)$ is \emph{proper }if $X\notin\C$.
Another cover $\mathcal{S}$ is \emph{compatible with }$\C$ if the
ideal generated by $\C\cup\mathcal{S}$ is proper.

Of course two ideal covers $\C$ and $\mathcal{S}$ ($\adh_{\xi}\C_{c}=\adh_{\xi}\mathcal{S}_{c}=\emptyset$)
are compatible if and only if $\C_{c}\#\mathcal{S}_{c}$. 

Given a class of filters $\mathbb{D}$, we say that a cover $\C$
is \emph{of class }$\mathbb{D}_{*}$ if $(\C^{\cup})_{c}\in\mathbb{D}$.
\begin{cor}
\label{prop:CVDclosed} The following are equivalent:
\begin{enumerate}
\item $\mathbb{U}_{\xi}X$ is $\mathbb{D}^{*}$-closed;
\item $\mathbb{K}_{\mathbb{D}}(\xi)=\mathbb{K}_{\mathbb{F}}(\xi)$, that
is, every $\mathbb{D}$-compactoid filter on $\xi$ is compactoid;
\item $\xi$ is $\mathbb{K}_{\mathbb{D}}$-compact;
\item $\xi$ is $(\mathbb{K}_{\mathbb{D}}\cap\mathbb{U})$-compact;
\item For every proper ideal cover, there is a compatible proper cover of
class $\mathbb{D}_{*}$.
\end{enumerate}
\end{cor}

\begin{proof}
$(1)\iff(2)$ is Proposition \ref{prop:clDisClJUxiX}. Assume $(2)$.
If $\F\in\mathbb{K}_{\mathbb{D}}$ then $\F$ is compactoid and in
particular $\adh_{\xi}\F\neq\emptyset$. Thus $\xi$ is $\mathbb{K}_{\mathbb{D}}$-compact.
That $(3)\then(4)$ is obvious. On the other hand, $(4)\then(2)$
because if $\F\in\mathbb{K}_{\mathbb{D}}$ then $\U\in\mathbb{K}_{\mathbb{D}}\cap\mathbb{U}$
for every $\U\in\beta\F$ and thus $\lim_{\xi}\U\neq\emptyset$, that
is, $\F\in\mathbb{K}_{\mathbb{F}}$. 

Finally, $(3)\iff(5)$ because $(3)$ can be rephrased in terms of
covers: 
\begin{eqnarray*}
\left(\F\in\mathbb{K}_{\mathbb{D}}\then\adh_{\xi}\F\neq\emptyset\right) & \iff & \left(\adh_{\xi}\F=\emptyset\then\F\notin\mathbb{K}_{\mathbb{D}}\right)\\
 & \iff & \left(\adh_{\xi}\F=\emptyset\then\exists\D\in\mathbb{D}:\D\#\F,\adh_{\xi}\D=\emptyset\right),
\end{eqnarray*}
and the latest condition rephrases as $(5)$ because a proper ideal
cover $\C$ satisfies $\F=\C_{c}\in\mathbb{F}X$ with $\adh_{\xi}\F$,
and if $\D\in\mathbb{D}$ with $\D\#\F$ and $\adh_{\xi}\D=\emptyset$
then $\mathcal{S}=\D_{c}$ is a compatible cover in $\mathbb{D}_{*}$. 
\end{proof}
Recall that $\mathbb{K}_{\mathbb{F}_{0}}=\mathbb{F}$, so that $\mathbb{U}_{\xi}X$
is closed (hence compact) if and only if $\xi$ is compact. 
\begin{thm}
Let $\mathbb{D}\subset\mathbb{J}$ be two classes of filters and let
$\xi$ be a convergence. If $\mathbb{U}_{\xi}X$ is cover-$\nicefrac{\mathbb{J_{*}}}{\mathbb{D}_{*}}$-compact
and $\mathbb{J}$ is $\beta$-compatible then $\xi$ is cover-$\nicefrac{\mathbb{J_{*}}}{\mathbb{D}_{*}}$-compact.
If $\mathbb{J}=\mathbb{F}$ and $\mathbb{D}$ is $\beta$-compatible,
the converse is true.
\end{thm}

\begin{proof}
Let $\G\in\mathbb{J}$ with $\adh_{\xi}\G=\emptyset$. Then $\mathbb{U}_{\xi}X\subset(\beta\G)^{c}=\bigcup_{G\in\G}\beta(X\setminus G)$
and thus $\{\beta(X\setminus G):G\in\G\}$ is a cover of $\mathbb{U}_{\xi}X$
in $\mathbb{J}_{*}$ by $\beta$-compatibility of $\mathbb{J}$. As
$\mathbb{U}_{\xi}X$ is cover-$\nicefrac{\mathbb{J_{*}}}{\mathbb{D}_{*}}$-compact,
there is $\mathcal{S}\subset\{\beta(X\setminus G):G\in\G\}$ which
is a cover of $\mathbb{U}_{\xi}X$ of class $\mathbb{D}_{*}$, so
that $\{G\in\G:\beta(X\setminus G\}\in\mathcal{S}$ belongs to $\mathbb{D}$
and satisfies $\adh_{\xi}\D=\emptyset$. Thus $\xi$ is cover-$\nicefrac{\mathbb{J_{*}}}{\mathbb{D}_{*}}$-compact.

Assume now that $\xi$ is cover-$\nicefrac{\mathbb{F_{*}}}{\mathbb{D}_{*}}$-compact
and let $\C$ be an open cover of $\mathbb{U}_{\xi}X$. It has a refinement
of the form $\left\{ \beta A:A\in\A\right\} $ for some family $\A\subset\mathbb{P}X$,
which covers $\mathbb{U}_{\xi}X$. As $\beta(A\cup B)=\beta A\cup\beta B$,
we can assume $\A$ and $\{\beta A:A\in\A\}$ to be ideals. Hence
$\F=\A_{c}$ is a filter on $X$ with $\adh_{\xi}\F=\emptyset$. Moreover,
$\xi$ is cover-$\nicefrac{\mathbb{F_{*}}}{\mathbb{D}_{*}}$-compact
so that there is $\D\leq\F$ with $\adh_{\xi}\D=\emptyset$, equivalently,
$\{\beta(X\setminus D):D\in\D\}\subset\{\beta A:A\in\A\}$ is a cover
of $\mathbb{U}_{\xi}X$, which is in $\mathbb{D}_{*}$ by $\beta$-compatibility
of $\mathbb{D}$. Hence $\C$ has a refinement in $\mathbb{D}_{*}$
that covers $\mathbb{U_{\xi}}X$, which is thus cover-$\nicefrac{\mathbb{F_{*}}}{\mathbb{D}_{*}}$-compact.
\end{proof}
In particular,
\begin{cor}
\label{cor:UXLindelof} $\mathbb{U}_{\xi}X$ is Lindelöf if and only
if $\xi$ is cover-Lindelöf. 
\end{cor}

As a Lindelöf subset of a topological space is also $\delta$-closed,
we obtain:
\begin{cor}
\label{cor:Lindelof} If $\xi$ is cover-Lindelöf then $\mathbb{U}_{\xi}X$
is $\delta$-closed.
\end{cor}

The converse of Corollary \ref{cor:Lindelof} is false:
\begin{example}[A topology $\xi$ such that $\mathbb{U}_{\xi}X$ is $\delta$-closed
but not Lindelöf]
\label{exa:UXdeltaclosednotLindelof} Let $X$ be an uncountable
set and let $\xi$ be a prime topology on $X$ with non-isolated point
$\infty$ and $\N_{\xi}(\infty)=\{\infty\}\wedge(x_{n})_{n\in\omega}$,
where $\{x_{n}:n\in\omega\}$ is a free sequence. Then $\xi$ is not
Lindelöf and thus $\mathbb{U}_{\xi}X$ is not Lindelöf by Corollary
\ref{cor:UXLindelof}. On the other hand, if $\U\in\mathbb{K}_{\mathbb{F}_{1}}\cap\mathbb{U}^{\circ}X$
then $\H\#(x_{n})_{n\in\omega}$ for every $\H\in\mathbb{F}_{1}^{\circ}$
with $\H\leq\U$. If there was $k\in\omega$ with $\{x_{n}:n\geq k\}\notin\U$,
then $X\setminus\{x_{n}:n\geq k\}\in\U$ and, given $\H\in\mathbb{F}_{1}^{\circ}$,
$\H\leq\U$, the filter $\H^{'}=\H\vee(X\setminus\{x_{n}:n\geq k\})$
would be a countably based subfilter of $\U$ that does not mesh with
$(x_{n})_{n\in\omega}$. Hence $\U\geq(x_{n})_{n\in\omega}$ and $\U$
converges. In view of Proposition \ref{prop:CVDclosed} for $\mathbb{D}=\mathbb{F}_{1}$,
$\mathbb{U}_{\xi}X$ is $\delta$-closed.
\end{example}

We can transpose this simple example to $\mathbb{U}_{\xi}(x)$ to
the effect that there is:
\begin{example}[A convergence $\sigma$ with $\S\sigma=\S_{1}\sigma$ and a non cover-Lindelöf
singleton]
\label{exa:nonlindelofsingleton} Consider the set $S=\mathbb{U}_{\xi}X\subset\mathbb{U}X$
as in Example \ref{exa:UXdeltaclosednotLindelof} and consider on
$X$ the prime pseudotopology $\sigma$ defined by 
\[
\mathbb{U}_{\sigma}(\infty)=\{\infty\}^{\uparrow}\cup S.
\]

Then $\sigma=\S\sigma=\S_{1}\sigma$ because $S$ is $\delta$-closed.
On other hand, $\mathbb{U}_{\sigma}(\infty)$ is not Lindelöf (because
$S$ is not) and thus $\{\infty\}$ is not cover-Lindelöf by Corollary
\ref{cor:Lindelofptwise}.
\end{example}

\begin{prop}
\label{prop:clFsigmaforUxiX} The following are equivalent:
\begin{enumerate}
\item there is an $F_{\sigma}$ subset $S$ of $\mathbb{U}_{\xi}X$ such
that $\cl_{\beta}S=\cl_{\beta}(\mathbb{U}_{\xi}X)=\mathbb{U}X$;
\item There is a sequence $(\D_{n})_{n\in\omega}$ of compactoid filters
such that $\beta\left(\bigwedge_{n\in\omega}\D_{n}\right)=\mathbb{U}X$,
equivalently, 
\[
\bigwedge_{n\in\omega}\D_{n}=\{X\}.
\]
\end{enumerate}
\end{prop}

\begin{proof}
Note that since $\left\{ \left\{ x\right\} ^{\uparrow}:x\in X\right\} \subset\mathbb{U}_{\xi}X$,
the set $\mathbb{U_{\xi}}X$ is dense in $(\mathbb{U}X,\beta)$, that
is, $\cl_{\beta}(\mathbb{U}_{\xi}X)=\mathbb{U}X.$

By definition, there is an $F_{\sigma}$ subset $S$ of $\mathbb{U}_{\xi}X$
such that $\cl_{\beta}S=\cl_{\beta}(\mathbb{U}_{\xi}X)$ if and only
if there is a sequence $(\D_{n})_{n\in\omega}$ of filters with $\beta\D_{n}\subset\mathbb{U}_{\xi}X$
and $\cl_{\beta}\left(\bigcup_{n\in\omega}\beta\D_{n}\right)=\cl_{\beta}\mathbb{U}_{\xi}X$.
The conclusion follows from (\ref{thm:deltaclosednotclosureofFsigma}).
\end{proof}

\begin{cor}
\label{cor:sigmacompact} Let $(X,\xi)$ be a $\T$-regular pseudotopological
space. Then the following are equivalent:
\begin{enumerate}
\item $\xi$ is $\sigma$-compact;
\item $\mathbb{U}_{\xi}X$ is $\sigma$-compact;
\item there is an $F_{\sigma}$ subset $S$ of $\mathbb{U}_{\xi}X$ with
$\cl_{\beta}S=\cl_{\beta}(\mathbb{U}_{\xi}X)=\mathbb{U}X$.
\end{enumerate}
\end{cor}

\begin{proof}
$(1)\then(2):$ If $X=\bigcup_{n\in\omega}K_{n}$ where each $K_{n}$
is compact. Let $\D_{n}=\O(K_{n})^{\uparrow}$. Then $\mathbb{U}_{\xi}X=\bigcup_{n\in\omega}\beta\D_{n}$
is $\sigma$-compact. Indeed, if $\U\in\mathbb{U}_{\xi}X$, there
is $x$ with $x\in\lim_{\xi}\U$ and thus there is $n\in\omega$ with
$x\in\lim_{\xi}\U\cap K_{n}$. Then $\U\geq\O(K_{n})=\D_{n}$ belongs
to $\beta\D_{n}$.

That $(2)\then(3)$ is clear. To see that $(3)\then(1)$, assume there
is an $S\subset\mathbb{U}_{\xi}X$ as in (3). In view of Proposition
\ref{prop:clFsigmaforUxiX}, there is a sequence $(\D_{n})_{n\in\omega}$
of compactoid filters such that $\bigwedge_{n\in\omega}\D_{n}=\{X\}.$
In view of Corollary \ref{cor:compactadherence}, each $\cl_{\xi}^{\natural}\D_{n}$
is a compact filter with compact kernel $K_{n}=\ker\cl_{\xi}^{\natural}\D_{n}$.
Moreover, $\bigwedge_{n\in\omega}\cl_{\xi}^{\natural}\D_{n}=\{X\}$
because $\cl_{\xi}^{\natural}\D_{n}\leq\D_{n}$ for each $n\in\omega,$
so that, in view of Lemma \ref{lem:kernels}, 
\[
X=\ker\bigwedge_{n\in\omega}\cl_{\xi}^{\natural}\D_{n}=\bigcup_{n\in\omega}\ker\cl_{\xi}^{\natural}\D_{n}=\bigcup_{n\in\omega}K_{n}.
\]
Therefore, $\xi$ is $\sigma$-compact.
\end{proof}
Theorem \ref{thm:deltaclosednotclosureofFsigma} now follows at once,
as there are regular Lindelöf topological spaces that are not $\sigma$-compact
(e.g., $\mathbb{R}^{\omega}$). For such a space $(X,\xi)$, the set
$\mathbb{U}_{\xi}X$ is $\delta$-closed (in fact, Lindelöf) by Corollary
\ref{cor:Lindelof}, but not closed because $\xi$ is not compact.
Moreover, by Corollary \ref{cor:sigmacompact}, there is no $F_{\sigma}$
subset $S$ of $\mathbb{U}_{\xi}X$ with $\cl_{\beta}S=\cl_{\beta}(\mathbb{U}_{\xi}X)$
because $\xi$ is not $\sigma$-compact.

Recall that a topological space is\emph{ hemicompact }if there is
a sequence of compact subsets $(K_{n})_{n\in\omega}$ such that for
every compact subset $K$, there is $n\in\omega$ with $K\subset K_{n}$.
Of course every hemicompact topology is $\sigma$-compact but not
conversely. 
\begin{prop}
\label{prop:UXhemiC} Let $\xi$ be a convergence. The following are
equivalent:
\begin{enumerate}
\item $\mathbb{U}_{\xi}X$ is hemicompact;
\item There is a sequence $(\D_{n})_{n\in\omega}$ of compactoid filter
such that
\[
\F\in\mathbb{K}_{\mathbb{F}}\then\exists n:\F\geq\D_{n}.
\]
\end{enumerate}
\end{prop}

\begin{proof}
$\mathbb{U}_{\xi}X$ is hemicompact if there is a sequence of compact
subsets, that is, a sequence of filters $\D_{n}$ with $\beta\D_{n}\subset\mathbb{U}_{\xi}X$
such that any compact subset $\beta\F\subset\mathbb{U}_{\xi}X$ for
$\F\in\mathbb{F}X$, there is $n\in\omega$ with $\beta\F\subset\beta\D_{n}$,
equivalently, $\D_{n}\leq\F.$ As $\beta\F\subset\mathbb{U}_{\xi}X$
if and only if $\F\in\mathbb{K}_{\mathbb{F}}$, the conclusion follows.
\end{proof}
\begin{cor}
\label{cor:UXhemicompactiffXhemicompact} If $\xi$ is a regular Hausdorff
topology then $\xi$ is hemicompact if and only if $\mathbb{U}_{\xi}X$
is hemicompact.
\end{cor}

\begin{proof}
Assume that $\xi$ is hemicompact and that $\{K_{n}\}_{n\in\omega}$
witnesses the definition. Let $\D_{n}=\O_{\xi}(K_{n})^{\uparrow}$.
Clearly, $\D_{n}$ is compact(oid) because $K_{n}$ is compact. If
now $\F$ is a compactoid filter, then in view of Corollary \ref{cor:compactadherence}
and Corollary \ref{cor:O(compact)}, $\adh(\cl^{\natural}\F)=\adh\F$
is compact, so that there is $n\in\omega$ with $\adh\F\subset K_{n}$
and thus 
\[
\O(K_{n})\subset\O(\adh\F)=\O(\cl_{\xi}^{\natural}\F)\leq\F
\]
 and thus $\D_{n}\leq\F$. In view of Proposition \ref{prop:UXhemiC},
$\mathbb{U}_{\xi}X$ is hemicompact.

Conversely, let $(\D_{n})_{n\in\omega}$ be as in Proposition \ref{prop:UXhemiC}
and let 
\[
K_{n}=\adh_{\xi}\D_{n}=\adh_{\xi}(\cl_{\xi}^{\natural}\D_{n}).
\]
 If $K$ is a compact subset of $|\xi|$ then $\{K\}^{\uparrow}$
is a compact filter, and thus $\{K\}\geq\D_{n}\geq\cl_{\xi}^{\natural}\D_{n}$
for some $n\in\omega$, that is, $K\subset\ker\cl_{\xi}^{\natural}\D_{n}=\adh_{\xi}\D_{n}=K_{n}$.
\end{proof}
The following two propositions are straightforward and characterize
the closure of $\mathbb{U}_{\xi}X$ under the closures involved in
the definitions of countable depth and countable depth for ultrafilters:
\begin{prop}
Let $\xi$ be a convergence. The following are equivalent:
\begin{enumerate}
\item $\cl_{\beta}S\subset\mathbb{U}_{\xi}X$ whenever $S$ is an $F_{\sigma}$-subset
of $\mathbb{U}_{\xi}X$;
\item If $(\D_{n})_{n\in\omega}$ is a sequence of compactoid filters, then
$\bigwedge_{n\in\omega}\D_{n}$ is compactoid.
\end{enumerate}
\end{prop}

\begin{prop}
Let $\xi$ be a convergence. The following are equivalent:
\begin{enumerate}
\item $\mathbb{U}_{\xi}X=\cl_{t}(\mathbb{U}_{\xi}X)$ ;
\item If $(\U_{n})_{n\in\omega}$ is a sequence of convergent ultrafilters,
then $\bigwedge_{n\in\omega}\U_{n}$ is compactoid.
\end{enumerate}
\end{prop}

\subsection{Topological properties of $\mathbb{U}X\setminus\mathbb{U}_{\xi}X$}

Note that
\[
\U\in\cl_{\beta}(\mathbb{U}X\setminus\mathbb{U}_{\xi}X)\iff\forall U\in\U\exists\W\in\beta U:\lm_{\xi}\W=\emptyset,
\]
that is, $\U\in\cl_{\beta}(\mathbb{U}X\setminus\mathbb{U}_{\xi}X)$
if and only if $\U$ has no compactoid element, equivalently, every
co-compactoid set is in $\U$. In other words,
\begin{equation}
\cl_{\beta}(\mathbb{U}X\setminus\mathbb{U}_{\xi}X)=\beta(\mathcal{K}_{c}).\label{eq:cocompactoid}
\end{equation}

More generally,

\[
\U\in\cl_{\mathbb{D}^{*}}(\mathbb{U}X\setminus\mathbb{U}_{\xi}X)\iff\forall\D\in\mathbb{D}:\D\leq\U,\beta\D\cap(\mathbb{U}X\setminus\mathbb{U}_{\xi}X)\neq\emptyset,
\]
that is, no $\mathbb{D}$-filter meshing with $\U$ is compactoid.
In particular, $\mathbb{U}X\setminus\mathbb{U}_{\xi}X$ is $\mathbb{D}^{*}$-closed
if such ultrafilters do not converge, equivalently, if
\[
\lm_{\xi}\U\neq\emptyset\then\exists\D\in\mathbb{D}\;\D\leq\U\;\text{ and }\D\text{ is compactoid.}
\]

As a result:
\begin{prop}
\label{prop:nonCVDstarclosed} $\mathbb{U}X\setminus\mathbb{U}_{\xi}X$
is $\mathbb{D}^{*}$-closed if and only if 
\[
\lm_{\xi}\U\neq\emptyset\then\exists\D\in\mathbb{D}\;\D\leq\U\;\text{ and }\D\text{ is compactoid.}
\]
\end{prop}

More generally,
\begin{prop}
\label{prop:equalclosuresnonCVu} Let $\mathbb{D}\subset\mathbb{J}$
be two classes of filters. The following are equivalent:
\begin{enumerate}
\item $\cl_{\mathbb{D}^{*}}(\mathbb{U}X\setminus\mathbb{U}_{\xi}X)=\cl_{\mathbb{J}^{*}}(\mathbb{U}X\setminus\mathbb{U}_{\xi}X)$;
\item For every $\U\in\mathbb{U}X$,
\[
\exists\G\in\mathbb{J}\cap\mathbb{K}_{\mathbb{F}},\G\leq\U\then\exists\D\in\mathbb{D}\cap\mathbb{K}_{\mathbb{F}},\D\leq\U.
\]
\end{enumerate}
\end{prop}

Recall that $\mathbb{F}_{0}^{*}$ is the usual topology $\beta$ of
$\mathbb{U}X$. 
\begin{prop}
Let $\mathbb{J}$ be a class of filters. The following are equivalent:
\begin{enumerate}
\item $\cl_{\beta}(\mathbb{U}X\setminus\mathbb{U}_{\xi}X)=\cl_{\mathbb{J}^{*}}(\mathbb{U}X\setminus\mathbb{U}_{\xi}X)$;
\item For every $\U\in\mathbb{U}X$,
\[
\exists\G\in\mathbb{J}\cap\mathbb{K}_{\mathbb{F}},\G\leq\U\then\exists K\in\mathbb{\mathbb{F}}_{0}\cap\mathbb{K}_{\mathbb{F}},K\in\U;
\]
\item Every compactoid filter in $\mathbb{J}$ has a compactoid element.
\end{enumerate}
\end{prop}

\begin{proof}
$(1)\iff(2)$ is Proposition \ref{prop:equalclosuresnonCVu} for $\mathbb{D}=\mathbb{F}_{0}$
and $(3)\then(2)$ is clear. For $(2)\then(3)$, let $\G\in\mathbb{J}\cap\mathbb{K}_{\mathbb{F}}$.
For every $\U\in\beta\G$, there is by (2) a compactoid set $K_{\U}\in\U$.
By compactness of $\beta\G$, there is a finite subset $F$ of $\beta\G$
with $\bigcup_{\U\in F}K_{\U}\in\G$, and $\bigcup_{\U\in F}K_{\U}$
is compactoid as a finite union of compactoid sets. 
\end{proof}
In particular, $\cl_{\beta}(\mathbb{U}X\setminus\mathbb{U}_{\xi}X)\subset\cl_{\delta}(\mathbb{U}X\setminus\mathbb{U}_{\xi}X)$
if and only if every countably based compactoid filter has a compactoid
element. In this case, countably based convergent filter contain a
compactoid sets, and thus:

\begin{cor}
If $\cl_{\beta}(\mathbb{U}X\setminus\mathbb{U}_{\xi}X)\subset\cl_{\delta}(\mathbb{U}X\setminus\mathbb{U}_{\xi}X)$
then $\I_{1}\xi$ is locally $\xi$-compactoid.
\end{cor}

As a result of Proposition \ref{prop:nonCVDstarclosed} for $\mathbb{D}=\mathbb{F}_{0}$,
we obtain \cite{D.covers,myn.completeness} (\footnote{This result can be found in \cite{D.covers} and was revisited in
\cite{myn.completeness}. A proof is included here for completeness
and to illustrate an instance of Proposition \ref{prop:nonCVDstarclosed}.}):
\begin{cor}
\label{cor:nonCVclosed} The following are equivalent:
\begin{enumerate}
\item $\mathbb{U}X\setminus\mathbb{U}_{\xi}X$ is $\beta$-closed;
\item $\mathbb{U}_{\xi}X$ is $\beta$-open;
\item $\xi$ is locally compactoid;
\item $\adh_{\xi}\mathcal{K}_{c}=\emptyset$.
\end{enumerate}
\end{cor}

\begin{proof}
That $(1)\iff(2)$ is obvious. 

$(1)\then(3)$: Let $\F$ be a convergent filter. For every $\U\in\beta\F$,
there is, by Proposition \ref{prop:nonCVDstarclosed}, a compactoid
set $K_{\U}\in\U$. By compactness of $\beta\F$, there is a finite
subset $S$ of $\beta\F$ with $\bigcup_{\U\in S}K_{\U}\in\F$. As
$\bigcup_{\U\in S}K_{\U}$ is a finite union of compactoid sets, it
is compactoid. 

$(3)\then(4):$ If $\xi$ is locally compactoid, a convergent filter
cannot mesh with $\mathcal{K}_{c}$. 

$(4)\then(1):$ If $\U\in\mathbb{U}X$ and $\lim_{\xi}\U\neq\emptyset$
and $\adh_{\xi}\mathcal{K}_{c}=\emptyset$ then there is $H\in\mathcal{K}_{c}$
with $H\notin\U^{\#}=\U$, so that $H^{c}$, which is compactoid,
belongs to $\U$. In view of Proposition \ref{prop:nonCVDstarclosed},
we obtain (1).
\end{proof}
\begin{cor}
Let $\xi$ be a non-compact topology. Then $\xi$ is locally compactoid
if and only if 
\[
\cl_{\xi}^{\natural}\mathcal{K}_{c}=\mathcal{K}_{c}.
\]
\end{cor}

\begin{proof}
If $\xi$ is locally compactoid, then $\adh_{\xi}\mathcal{K}_{c}=\emptyset$
by Corollary \ref{cor:nonCVclosed}, so that, in view of Lemma \ref{lem:consonantcovers},
$\cl_{\xi}^{\natural}\mathcal{K}_{c}\geq\mathcal{\mathcal{K}}_{c}$.
The reverse inequality is always true and thus $\cl_{\xi}^{\natural}\mathcal{K}_{c}=\mathcal{K}_{c}$.
Conversely, if $\cl_{\xi}^{\natural}\mathcal{K}_{c}=\mathcal{K}_{c}$
then $\adh_{\xi}\mathcal{K}_{c}=\emptyset$ by Lemma \ref{lem:consonantcovers},
and thus $\xi$ is locally compactoid by Corollary \ref{cor:nonCVclosed}.
\end{proof}
Recall that $\mathbb{F}_{1}^{*}$ is the $G_{\delta}$-topology of
$(\mathbb{U}X,\beta)$. 

As a result of Proposition \ref{prop:nonCVDstarclosed} and Corollary
\ref{cor:PCTandbik}:
\begin{cor}
\label{cor:bik} If $\xi$ is a regular Hausdorff topology, then $\mathbb{U}X\setminus\mathbb{U}_{\xi}X$
is $\delta$-closed if and only if $\xi$ is bi-$k$. 
\end{cor}

\begin{cor}
\label{cor:biseq} If $\xi$ is bisequential then $\mathbb{U}X\setminus\mathbb{U}_{\xi}X$
is $\delta$-closed.
\end{cor}

As there are bi-$k$ spaces that are not bisequential, the converse
of Corollary \ref{cor:biseq} is not true but in the case of a prime
convergence on a countable set, it is:
\begin{thm}
Let $\xi$ be a countable prime convergence. The following are equivalent:
\begin{enumerate}
\item $\xi$ is bisequential;
\item $\mathbb{U}X\setminus\mathbb{U}_{\xi}X$ is $\delta$-closed;
\item $\mathbb{}\mathbb{U}_{\xi}X$ is $\delta$-open.
\end{enumerate}
\end{thm}

\begin{proof}
That $(2)\iff(3)$ is obvious and $(1)\then(2)$ is Corollary \ref{cor:biseq}.

To see that $(3)\then(1)$, let $x_{0}$ denote the non-isolated point
of $\xi.$If $x_{0}\in\lim_{\xi}\U$ for some free ultrafilter $\U$,
then $\U\in\mathbb{U}_{\xi}X$, which is $\delta$-open so that, in
view of Lemma \ref{lem:GdeltaBase}, there is $\D\in\mathbb{F}_{1}$
with 
\[
\U\in\beta\D=\beta\D^{\bullet}\cup\beta\D^{\circ}\subset\mathbb{U}_{\xi}X.
\]

If $\U\geq\D^{\circ}$, we have $\D^{\circ}\in\mathbb{F}_{1}$ and
$x_{0}\in\lim_{\S\xi}\D^{\circ}$ because convergent free ultrafilters
can only converge to $x_{0}$. Hence $x\in\lim_{\S\I_{1}\xi}\U$.

If $\U\geq\D^{\bullet}$ then $\ker\D\in\U$. As $\U$ is free, $\ker\D$
is infinite, and all free ultrafilters on $\ker\D$ converge to $x_{0}$
so that $x_{0}\in\lim_{\S\xi}(\ker\D)_{0}$. Because $|\xi|$ is countable,
$(\ker\D)_{0}\in\mathbb{F}_{1}$ and thus $x_{0}\in\lim_{\S\I_{1}\xi}\U$.
\end{proof}
\begin{thm}
\label{thm:UXminusUxiXcompactness} Let $\mathbb{D}\subset\mathbb{J}$
be two classes of filters, where $\mathbb{J}$ is $\beta$-compatible.
If $\mathbb{U}X\setminus\mathbb{U}_{\xi}X$ is $\nicefrac{\mathbb{J_{*}}}{\mathbb{D}_{*}}$-compact
then 
\begin{equation}
\G\in\mathbb{J}\cap\mathbb{K}_{\mathbb{F}}\then\exists\D\in\mathbb{D}\cap\mathbb{K}_{\mathbb{F}}:\D\leq\G.\label{eq:bettercompact}
\end{equation}
Moreover, if $\mathbb{J}=\mathbb{F}$ and $\mathbb{D}$ is $\beta$-compatible,
the converse is true.
\end{thm}

\begin{proof}
Suppose $\mathbb{U}X\setminus\mathbb{U}_{\xi}X$ is $\nicefrac{\mathbb{J_{*}}}{\mathbb{D}_{*}}$-compact
and let $\G\in\mathbb{K}_{\mathbb{F}}\cap\mathbb{J}$. Then $\beta\G\cap(\mathbb{U}X\setminus\mathbb{U}_{\xi}X)=\emptyset$,
that is,
\[
\mathbb{U}X\setminus\mathbb{U}_{\xi}X\subset\mathbb{U}X\setminus\beta\G=\bigcup_{G\in\G}\beta(X\setminus G).
\]
By $\beta$-compatibility of $\mathbb{J}$, $\{\beta(X\setminus G):G\in\G\}$
is a cover of $\mathbb{U}X\setminus\mathbb{U}_{\xi}X$ in $\mathbb{J}_{*}$.
By $\nicefrac{\mathbb{J_{*}}}{\mathbb{D}_{*}}$-compactness, there
is $\D\in\mathbb{D}$ with $\D\leq\G$ and 
\[
\mathbb{U}X\setminus\mathbb{U}_{\xi}X\subset\mathbb{U}X\setminus\beta\D=\bigcup_{D\in\D}\beta(X\setminus D),
\]
so that $\D\in\mathbb{K}_{\mathbb{F}}\cap\mathbb{D}$.

Assume now (\ref{eq:bettercompact}) in the case $\mathbb{J}=\mathbb{F}$.
Any open cover $\C$ of $\mathbb{U}X\setminus\mathbb{U}_{\xi}X$ has
a refinement of the form $\{\beta A:A\in\A\}$ where $\A$ is an ideal,
that is also a cover. Let $\F$ be the filter $\A_{c}$. Then
\[
\mathbb{U}X\setminus\mathbb{U}_{\xi}X\subset\bigcup_{A\in\A}\beta A=\mathbb{U}X\setminus\beta\F,
\]
so that $\F\in\mathbb{K}_{\mathbb{F}}$. By (\ref{eq:bettercompact}),
there is $\D\in\mathbb{D\cap\mathbb{K}_{\mathbb{F}}}$ with $\D\leq\F$.
Hence 
\[
\mathbb{U}X\setminus\mathbb{U}_{\xi}X\subset\mathbb{U}X\setminus\beta\D,
\]
and by $\beta$-compatibility of $\mathbb{D}$, $\{\beta(X\setminus D):D\in\D\}$
is a subcover of $\{\beta A:A\in\A\}$ of class $\mathbb{D}_{*}$.
\end{proof}
\begin{cor}
\label{cor:nonCVlindelof} Let $(X,\xi)$ be a convergence space.
Then $\mathbb{U}X\setminus\mathbb{U}_{\xi}X$ is Lindelöf if and only
if every compactoid filter contains a countably based compactoid filter.
\end{cor}

As there are bi-$k$ regular Hausdorff topological spaces that are
not of pointwise countable type (see \cite{quest}), we obtain:
\begin{cor}
\label{cor:deltaclosednotLindelof} There is $X$ and a subset $S$
of $\mathbb{U}^{\circ}X$ which is $\delta$-closed but not Lindelöf.
\end{cor}

\begin{proof}
For such a topological space $(X,\xi)$, $S=\mathbb{U}X\setminus\mathbb{U}_{\xi}X$
is $\delta$-closed because $\xi$ is bi-$k$ by Corollary \ref{cor:PCTandbik}.
As $\xi$ is not of pointwise countable type, there is by Corollary
\ref{cor:bik} a convergent, hence compactoid, filter that does not
contain any countably based compactoid filter. In view of Corollary
\ref{cor:nonCVlindelof}, $S$ is not Lindelöf.
\end{proof}
\begin{cor}
\label{cor:metricnonCVLindelof} If $(X,\xi)$ is of countable type
(\footnote{that is, every compact set is contained in a compact set of countable
character. Clearly, metrizable spaces are of countable type.}) then $\mathbb{U}X\setminus\mathbb{U}_{\xi}X$ is Lindelöf. 
\end{cor}

\begin{proof}
If $\F$ is compactoid then in view of Corollary \ref{cor:compactadherence}
and Corollary \ref{cor:O(compact)},
\[
\F\geq\O(\cl^{\natural}\F)=\O(\adh\F)
\]
and $\adh\F$ is compact. As $\xi$ is of countable type, there is
a compact set $K$ with $\O(\adh\F)\geq\O(K)$ and $\O(K)$ is a countably
based compact filter.
\end{proof}
Finally note that $\sigma$-compactness of $\mathbb{U}X\setminus\mathbb{U}_{\xi}X$
has been characterized in terms of completeness. A family $\mathbb{D}$
of non-adherent filters on $(X,\xi)$ is \emph{cocomplete }if for
every $\G\in\mathbb{F}X$
\[
\adh_{\xi}\G=\emptyset\then\exists\D\in\mathbb{D}\;\G\#\D.
\]

A convergence $\xi$ is \emph{countably complete} if it admits a countable
cocomplete family of filters. A completely regular topological space
is countably complete if and only if it is \v Cech-complete. See
\cite{D.covers,myn.completeness} for details.
\begin{prop}
\cite[Corollary 7]{myn.completeness,D.covers}\label{prop:complete}
$\mathbb{U}X\setminus\mathbb{U}_{\xi}X$ is $\sigma$-compact if and
only if $\xi$ is countably complete .
\end{prop}

\begin{cor}
\label{cor:Lindelofnotsigmacompact} There is a set $X$ and a subset
$S$ of $\mathbb{U^{\circ}}X$ that is Lindelöf but not $\sigma$-compact.
\end{cor}

\begin{proof}
Let $(X,\tau)$ be a non complete metric space (such as $\mathbb{Q}$
with its usual topology). Then, in view of Corollary \ref{cor:metricnonCVLindelof}
and Proposition \ref{prop:complete}, $S=\mathbb{U}X\setminus\mathbb{U}_{\tau}X$
is as desired.
\end{proof}
Similarly, hemicompactness of $\mathbb{U}X\setminus\mathbb{U}_{\xi}X$
has been characterized in terms of \emph{ultracompleteness}. A family
$\mathbb{D}$ of non-adherent filters on $(X,\xi)$ is \emph{ultracocomplete
}if for every $\G\in\mathbb{F}X$
\[
\adh_{\xi}\G=\emptyset\then\exists\D\in\mathbb{D}\;\G\geq\D.
\]

A convergence $\xi$ is \emph{countably ultracomplete} if it admits
a countable ultracocomplete family of filters. See \cite{myn.completeness}
for details.
\begin{prop}
\cite[Corollary 9]{myn.completeness}\label{prop:ultracomplete} A
convergence $\xi$ is countably ultracomplete if and only if $\mathbb{U}X\setminus\mathbb{U}_{\xi}X$
is hemicompact.
\end{prop}

\begin{cor}
\label{cor:hemiCnotclosed} \label{cor:hemicompactnotclosed} There
is a set $X$ and a subset $S$ of $\mathbb{U^{\circ}}X$ that is
hemicompact but not $\beta$-closed.
\end{cor}

\begin{proof}
If $(X,\tau)$ is countably ultracomplete but not locally compactoid
(e.g., \cite[Example 3.1]{bijagoar2001ultracomplete}), then in view
of Corollary \ref{cor:nonCVclosed} and Proposition \ref{prop:ultracomplete},
$S=\mathbb{U}X\setminus\mathbb{U}_{\tau}X$ is as desired.
\end{proof}
Finally, let us note that similarly to Corollary \ref{cor:sigmacompact},
$\sigma$-compactness of $\mathbb{U}X\setminus\mathbb{U}_{\xi}X$
is equivalent to having a $\sigma$-compact dense subset.
\begin{prop}
\label{prop:FsigmaclosureofnonCV} The following are equivalent:
\begin{enumerate}
\item there is an $F_{\sigma}$ subset $S$ of $\mathbb{U}X\setminus\mathbb{U}_{\xi}X$
such that 
\[
\cl_{\beta}S=\cl_{\beta}(\mathbb{U}X\setminus\mathbb{U}_{\xi}X);
\]
\item there is a countable family $(\D_{i})_{i\in\omega}$ of non-adherent
filters such that 
\[
\bigwedge_{i\in\omega}\D_{i}=\mathcal{K}_{c};
\]
\item there is a family $(\C_{i})_{i\in\omega}$ of ideal covers such that
whenever one selects $C_{i}\in\C_{i}$ for every $i\in\omega$, the
set $\bigcap_{i\in\omega}C_{i}$ is a compactoid set.
\end{enumerate}
\end{prop}

\begin{proof}
If (1) there is a sequence $(\D_{i})_{i\in\omega}$ of non-adherent
filters such that 
\[
\cl_{\beta}(\mathbb{U}X\setminus\mathbb{U}_{\xi}X)\subset\cl_{\beta}\{\D_{n}:n\in\omega\},
\]
that is, in view of (\ref{eq:cocompactoid}), 
\[
\beta\mathcal{K}_{c}\subset\beta(\bigwedge_{i\in\omega}\D_{i}).
\]
Since $\adh\D_{i}=\emptyset$, $\D_{i}\geq\mathcal{K}_{c}$ for every
$i$, hence $\bigwedge_{i\in\omega}\D_{i}\geq\mathcal{K}_{c}$ and
thus $\beta\mathcal{K}_{c}=\beta(\bigwedge_{i\in\omega}\D_{i}),$
equivalently, $\bigwedge_{i\in\omega}\D_{i}=\mathcal{K}_{c}$.

$(2)\then(3):$ Take $\C_{i}=(\D_{i})_{c}$. If one selects $C_{i}\in\C_{i}$,
equivalently, $D_{i}=C_{i}^{c}\in\D_{i}$ then $\bigcup_{i\in\omega}D_{i}\in\mathcal{K}_{c}$
so that there is a compactoid set $K^{c}\subset\bigcup_{i\in\omega}D_{i}$
equivalently, $\bigcap C_{i}\subset K$ so that $\bigcap C_{i}$ is
compactoid. 

$(3)\then(1)$ Let $\D_{i}=(\C_{i})_{c}\in\mathbb{F}X$, where $(\C_{i})_{i\in\omega}$
is as in (3). We show that 
\[
\cl_{\beta}(\bigcup_{i\in\omega}\beta\D_{i})=\beta(\bigwedge_{i\in\omega}\D_{i})=\mathbb{U}X\setminus\mathbb{U}_{\xi}X.
\]
 Indeed, each $\beta\D_{i}\subset\mathbb{U}X\setminus\mathbb{U}_{\xi}X$
because each $\C_{i}$ is a cover of $(X,\xi)$. If $\U\notin\beta(\bigwedge_{i\in\omega}\D_{i})$,
there is a selection $D_{i}\in\D_{i}$ with $\bigcup_{i\in\omega}D_{i}\notin\U$,
equivalently a selection $C_{i}=X\setminus D_{i}\in\C_{i}$ with $\bigcap_{i\in\omega}C_{i}\in\U$.
Because $\bigcap_{i\in\omega}C_{i}$ is compactoid we conclude that
$\lim_{\xi}\U\neq\emptyset$, that is, $\U\in\mathbb{U}_{\xi}X$,
that is, $\U\notin\mathbb{U}X\setminus\mathbb{U}_{\xi}X$.
\end{proof}

\subsection{Summary of results in this section}

The table below summarizes the main results of the section, together
with a few more easily verified characterizations that may prove useful
for future reference.\medskip{}

\begin{center}
\begin{table}[H]
\begin{centering}
\resizebox{\textwidth}{!}{%
\begin{tabular}{|c||c|c|c|}
\hline 
property & $\mathbb{U}_{\xi}(x)$ & $\mathbb{U}_{\xi}X$ & $\mathbb{U}X\setminus\mathbb{U}_{\xi}X$\tabularnewline
\hline 
\hline 
$\beta$-closed (compact) & $\S\xi=\S_{0}\xi$ & $\xi$ compact & $\xi$ locally compactoid\tabularnewline
\hline 
$\beta$-open & finitely generated & $\xi$ locally compactoid & $\xi$ compact\tabularnewline
\hline 
\multirow{3}{*}{$\delta$-closed} & \multirow{3}{*}{$\S\xi=\S_{1}\xi$} &  & $\lim_{\xi}\U\neq\emptyset\Longrightarrow$\tabularnewline
 &  & $\xi$ is $\mathbb{K}_{\mathbb{F}_{1}}$-compact & $\exists\D\in\mathbb{F}_{1}\cap\mathbb{K}_{\mathbb{F}}$,$\D\leq\U$\tabularnewline
 &  & ($\mathbb{K}_{\mathbb{F}}=\mathbb{K}_{\mathbb{F}_{1}}$) & $\iff\xi$ bi-$k$(reg. H. top.)\tabularnewline
\hline 
\multirow{2}{*}{$\delta$-open} & \multirow{2}{*}{bisequential} & $\lim_{\xi}\U\neq\emptyset\Longrightarrow$ & $\xi$ is $\mathbb{K}_{\mathbb{F}_{1}}$-compact\tabularnewline
 &  & $\exists\D\in\mathbb{F}_{1}\cap\mathbb{K}_{\mathbb{F}}$,$\D\leq\U$ & ($\mathbb{K}_{\mathbb{F}}=\mathbb{K}_{\mathbb{F}_{1}}$)\tabularnewline
\hline 
\multirow{2}{*}{$(\mathbb{F}_{\wedge1})^{*}$-closed} & \multirow{2}{*}{$\S\xi=\S_{\ensuremath{\wedge1}}\xi$} & $\xi$ is $\mathbb{K}_{\mathbb{F}_{\wedge1}}$-compact & $\lim_{\xi}\U\neq\emptyset\Longrightarrow$\tabularnewline
 &  & ($\mathbb{K}_{\mathbb{F}}=\mathbb{K}_{\mathbb{F}_{\wedge1}}$) & $\exists\D\in\mathbb{F}_{\wedge1}\cap\mathbb{K}_{\mathbb{F}}$,$\D\leq\U$\tabularnewline
\hline 
\multirow{2}{*}{$(\mathbb{F}_{\wedge1})^{*}$-open} & \multirow{2}{*}{$P$-space} & $\lim_{\xi}\U\neq\emptyset\Longrightarrow$ & $\xi$ is $\mathbb{K}_{\mathbb{F}_{\wedge1}}$-compact\tabularnewline
 &  & $\exists\D\in\mathbb{F}_{\wedge1}\cap\mathbb{K}_{\mathbb{F}}$,$\D\leq\U$ & ($\mathbb{K}_{\mathbb{F}}=\mathbb{K}_{\mathbb{F}_{\wedge1}}$)\tabularnewline
\hline 
$\exists S$ $F_{\sigma}$-subset: & $\exists(\D_{n})_{n\in\omega}\subset\lim_{\xi}^{-}(x)$: & $\exists(\D_{n})_{n\in\omega}$ in $\mathbb{K}_{\mathbb{F}}$: & \multirow{3}{*}{$\xi$ is $\omega$-$\mathbb{F}_{1}$-complete}\tabularnewline
\multirow{2}{*}{$\intr_{\delta}(\cdot)\subset S$} & $\H\in\mathbb{F}_{1}\cap\lim_{\xi}^{-}(x)\then$ & $\forall\H\in\mathbb{K_{F}\cap\mathbb{F}}_{1},\exists n\;\D_{n}\#\H$ & \tabularnewline
 & $\exists n:\H\#\D_{n}$ &  & \tabularnewline
\hline 
\multirow{2}{*}{$\sigma$-compact} & \multirow{2}{*}{$\mathsf{pp}(\xi)\leq\omega$} & $\xi$ is $\sigma$-compact & \multirow{2}{*}{$\xi$ is $\omega$-complete}\tabularnewline
 &  & ($\xi$ $\T$-reg. pseudotop.) & \tabularnewline
\hline 
\multirow{3}{*}{hemicompact} & \multirow{3}{*}{$\mathsf{p}(\xi,x)\leq\omega$} & $\exists(\D_{n})_{n\in\omega}$ in $\mathbb{K}_{\mathbb{F}}$: & \multirow{3}{*}{$\xi$ is $\omega$-ultracomplete}\tabularnewline
 &  & $\F\in\mathbb{K_{F}}\Longrightarrow\exists\D_{n}\leq\F$ & \tabularnewline
 &  & $\iff\xi$ hemicompact ($\xi$ top. H. reg.) & \tabularnewline
\hline 
\multirow{2}{*}{Lindelöf} & \multirow{2}{*}{$\{x\}$ cover-Lindelöf} & \multirow{2}{*}{$\xi$ is cover-Lindelöf} & $\F\in\mathbb{K_{F}\then}$\tabularnewline
 &  &  & $\exists\D\in\mathbb{F}_{1}\cap\mathbb{K}_{\mathbb{F}}$,$\D\leq\F$\tabularnewline
\hline 
\multirow{2}{*}{countably} & $x\notin\adh_{\xi}\bigvee_{n\in\omega}\D_{n}\then$ & $\adh_{\xi}\bigvee_{n\in\omega}\D_{n}=\emptyset\then$ & $\bigvee_{n\in\omega}\D_{n}\in\mathbb{K}_{\mathbb{F}}\then$\tabularnewline
 & $\exists n_{1},\ldots n_{k}:$ & $\exists n_{1},\ldots n_{k}:$ & $\exists n_{1},\ldots n_{k}:$\tabularnewline
compact & $x\notin\adh_{\xi}\bigvee_{k\in\{1,\ldots p\}}\D_{n_{k}}$ & $\adh_{\xi}\bigvee_{k\in\{1,\ldots p\}}\D_{n_{k}}=\emptyset$ & $\bigvee_{k\in\{1,\ldots p\}}\D_{n_{k}}\in\mathbb{K}_{\mathbb{F}}$\tabularnewline
\hline 
\multirow{2}{*}{$t$-closed} & \multirow{2}{*}{$\S\xi=\S\De^{\mathbb{U}}\xi$} & $\forall(\U_{n})_{n\in\omega}\in\mathbb{U}_{\xi}X$ & $\forall(\U_{n})_{n\in\omega}:\lim_{\xi}\U_{n}=\emptyset$\tabularnewline
 &  & $\bigwedge_{n\in\omega}\U_{n}\in\mathbb{K}_{\mathbb{F}}$ & $\adh_{\xi}\bigwedge_{n\in\omega}\U_{n}=\emptyset$\tabularnewline
\hline 
$S$ $F_{\sigma}$-subset$\Longrightarrow$ & \multirow{2}{*}{$\S\xi=\S\De\xi$} & $(\D_{n})_{n\in\omega}$ in $\mathbb{K}_{\mathbb{F}}$ & $\F_{n}\text{ s.t. }\adh_{\xi}\F_{n}=\emptyset\then$\tabularnewline
$\cl_{\beta}S$ subset &  & $\then\bigwedge_{n\in\omega}\D_{n}\in\mathbb{K}_{\mathbb{F}}$ & $\adh_{\xi}\bigwedge_{n\in\omega}\F_{n}=\emptyset$\tabularnewline
\hline 
$\exists S$ $F_{\sigma}$-subset: & \multirow{2}{*}{$\S_{0}\xi=\S\De\xi$} & $\xi$ is $\sigma$-compact & $\exists(\D_{n})_{n\in\omega}$, $\adh_{\xi}\D_{n}=\emptyset$,\tabularnewline
$\cl_{\beta}S=\cl_{\beta}(\cdot)$ &  & ($\xi$ $\T$-reg. pseudotop.) & $\bigwedge_{n\in\omega}\D_{n}=\mathcal{K}_{c}$\tabularnewline
\hline 
$\exists S$ $F_{\sigma}$-subset: & \multirow{2}{*}{$\S_{0}\I_{1}\xi\geq\S\De\xi\geq\S_{0}\xi$} & $\exists(\D_{n})_{n\in\omega}$ in $\mathbb{K}_{\mathbb{F}}$: & $\exists(\D_{n})_{n\in\omega}$, $\adh_{\xi}\D_{n}=\emptyset$,\tabularnewline
$\intr_{\delta}(\cdot)\subset\cl_{\beta}S$ &  & $\forall\H\in\mathbb{K_{F}\cap\mathbb{F}}_{1},\H\geq\bigwedge_{n\in\omega}\D_{n}$ & $\bigwedge_{n\in\omega}\D_{n}=(\mathcal{K}_{\omega})_{c}$\tabularnewline
\hline 
\multirow{2}{*}{$\cl_{\delta}(\cdot)=\cl_{\beta}(\cdot)$} & \multirow{2}{*}{$\S_{0}\xi=\S_{1}\xi$} & \multirow{2}{*}{$\mathbb{F}_{1}$-compact} & $\forall\F\in\mathbb{K}_{\mathbb{F}}\cap\mathbb{F}_{1},$\tabularnewline
 &  &  & $\exists F\in\F:\{F\}^{\uparrow}\in\mathbb{K_{F}}$\tabularnewline
\hline 
\end{tabular}}
\par\end{centering}
\medskip{}
\caption{\label{table:summaryUX}Summary of characterizations in $\mathbb{U}X$}
\end{table}
\par\end{center}

\section{Paving and pseudopaving numbers}

A family $\mathbb{D}$ of filters on a convergence space $(X,\xi)$
is a \emph{pavement at $x$ }if every $\D\in\mathbb{D}$ converges
to $x$ and, for every filter $\F$ converging to $x$, there is $\D\in\mathbb{D}$
with $\D\leq\F$. The family $\mathbb{D}$ is a \emph{pseudopavement
at $x$ }if every $\D\in\mathbb{D}$ converges to $x$ and, for every
ultrafilter $\U$ converging to $x$, there is $\D\in\mathbb{D}$
with $\U\in\beta\D$.

Let $\mathsf{p}(\xi,x)$ denote the \emph{paving number} of $\xi$
at $x$, that is, the smallest cardinality of a pavement at $x$ for
$\xi$, and let $\mathsf{pp}(\xi,x)$ denote the \emph{pseudopaving
number of $\xi$ at $x$}, that is, the smallest cardinality of a
pseudopavement at $x$ for $\xi$. Accordingly, the \emph{paving number
}and \emph{pseudopaving number }of $\xi$ are given by $\mathsf{p}(\xi)=\sup_{x\in|\xi|}\mathsf{p}(\xi,x)$
and $\mathsf{pp}(\xi)=\sup_{x\in|\xi|}\mathsf{pp}(\xi,x)$.
\begin{rem*}
Since in a convergence the filter $\{x\}^{\uparrow}$ always converges
to $x$, the paving and pseudopaving numbers are always at least 1
(and pretopologies are exactly the 1-paved convergences, equivalently,
the 1-pseudopaved convergences). 

Thus \cite[Theorem 26]{myn.completeness} and \cite[Theorem 28]{myn.completeness}
relating paving and pseudopaving numbers of a dual convergence to
the ultracompleteness and completeness numbers of the base convergence
are only valid for cardinality greater or equal to 1, though ultracompleteness
and completeness numbers can meaningfully be 0. This was not properly
spelled out in .\cite{myn.completeness}
\end{rem*}
\begin{lem}
\label{lem:pseudopavement} The following are equivalent for a family
$\mathbb{D}\subset\lim_{\xi}^{-}(x)$.
\begin{enumerate}
\item $\mathbb{D}$ is a pseudopavement at $x$;
\item For every filter $\F$ with $x\in\adh_{\xi}\F$, there is $\D\in\mathbb{D}$
with $\D\#\F$;
\item For every filter $\F$ with $x\in\lim_{\xi}\F$, there is $\D\in\mathbb{D}$
with $\D\#\F$;
\item 
\[
\mathbb{U}_{\xi}(x)=\bigcup_{\D\in\mathbb{D}}\beta\D.
\]
\end{enumerate}
\end{lem}

\begin{proof}
$(1)\iff(3)\iff(4)$ is \cite[Theorem 17]{myn.completeness}. Of course,
$(2)\then(3).$ To see that $(1)\then(2)$, let $x\in\adh_{\xi}\F$.
then there is $\U\in\beta\F$ with $x\in\lim_{\xi}\U$. Since $\mathbb{D}$
is a pseudopavement at $x$, there is $\D\in\mathbb{D}$ with $\U\geq\D$.
Thus $\F\#\D$. 
\end{proof}
\begin{cor}
\cite[Corollary 18]{myn.completeness}\label{cor:pseudopavingnumber}
Let $\xi$ be a convergence. Then $\mathsf{pp}(\xi,x)\leq\omega$
if and only if $\mathbb{U}_{\xi}(x)$ is $\sigma$-compact.
\end{cor}

\begin{cor}
\label{cor:cppimpliesLindelofpt} If $\xi$ is a convergence with
countable pseudopaving number then every singleton is cover-Lindelöf.
\end{cor}

\begin{proof}
If $\mathsf{pp}(\xi,x)\leq\omega$ then $\mathbb{U}_{\xi}(x)$ is
$\sigma$-compact, hence Lindelöf, and the conclusion follows from
Corollary \ref{cor:Lindelofptwise}.
\end{proof}
The converse is false:
\begin{example}[A convergence of uncountable pseudopaving number with cover-Lindelöf
singletons]
 \label{exa:ppuncountableLindelofpt} Take $X$ and $S$ as in Corollary
\ref{cor:Lindelofnotsigmacompact}. Let $\xi$ be a prime convergence
on $X$ with non-isolated point $x_{0}$, defined by 
\[
x_{0}\in\lm_{\xi}\F\iff\beta\F\subset S\cup\{\{x_{0}\}^{\uparrow}\}.
\]
 Since $S$ is Lindelöf, so is $S\cup\{\{x_{0}\}^{\uparrow}\}=\mathbb{U}_{\xi}(x_{0})$
and thus $\{x_{0}\}$ is cover-Lindelöf by Corollary \ref{cor:Lindelofptwise}.
On other hand, $\mathbb{U}_{\xi}(x_{0})=S\cup\{\{x_{0}\}^{\uparrow}\}$
is not $\sigma$-compact, so that $\mathsf{pp}(\xi)>\omega$ by Corollary
\ref{cor:pseudopavingnumber}.
\end{example}

Note also that
\begin{prop}
\cite[Corollary 20]{myn.completeness}\label{prop:paving} Let $\xi$
be a convergence. Then $\mathsf{p}(\xi,x)\leq\omega$ if and only
if $\mathbb{U}_{\xi}(x)$ is hemicompact.
\end{prop}

\begin{lem}
\label{lem:pseudopavvicinity} If $\mathbb{D}$ is a pseudopavement
of $\xi$ at $x$ then 
\[
\V_{\xi}(x)=\bigwedge_{\D\in\mathbb{D}}\D.
\]
\end{lem}

\begin{proof}
For every $\U\in\lim_{\xi}^{-}(x),$ there is $\D\in\mathbb{D}$ with
$\D\leq\U$ so that 
\[
\bigwedge_{\D\in\mathbb{D}}\D\leq\bigwedge_{\U\in\lim_{\xi}^{-}(x)}\U=\V_{\xi}(x).
\]
The reverse inequality follows from $\mathbb{D}\subset\lim_{\xi}^{-}(x)$.
\end{proof}
\begin{prop}
\label{prop:P1S0} If $\xi$ has countable pseudopaving number then
$\De\xi=\S_{0}\xi$. In particular, a convergence that is countably
deep and has countable pseudopaving number is a pretopology.
\end{prop}

\begin{proof}
If $\mathbb{D}$ is a countable pseudopavement at $x$, then $x\in\lim_{\De\xi}\bigwedge_{\D\in\mathbb{D}}\D$
and $\bigwedge_{\D\in\mathbb{D}}\D=\V_{\xi}(x)$ by Lemma \ref{lem:pseudopavvicinity}.
Thus, $x\in\lim_{\De\xi}\V_{\xi}(x)$ so that $\S_{0}\xi\geq\De\xi$.
The reverse inequality is always true.
\end{proof}
Maybe somewhat surprisingly in view of Corollary \ref{cor:sigmacompact}
and Proposition \ref{prop:FsigmaclosureofnonCV}, the converse is
not true:
\begin{example}[A convergence $\xi$ of uncountable pseudopaving number with $\De\xi=\S_{0}\xi$]
\label{exa:D1isS0butnotcpp} Let $\{D_{n}:n\in\omega\}$ be a partition
of $\omega$ into infinite subsets and for each $n\in\omega$ let
$\D(n)$ be an ultrafilter with $D_{n}\in\D(n)$. Define on $\omega\cup\{\infty\}$
the prime convergence $\xi$ given by 
\[
\infty\in\lm_{\xi}\F\iff\exists\;\U\in\mathbb{U}\omega:\F=\D(\U)=\bigcup_{U\in\U}\bigcap_{n\in U}\D(n).
\]
 As a contour of ultrafilters along an ultrafilter is an ultrafilter,
$\lim_{\xi}^{-}(\infty)\subset\mathbb{U}\omega$. Therefore, the only
pseudopavement of $\xi$ at $\infty$ is $\lim_{\xi}^{-}(\infty)$.
Moreover, using the notion of types of Frolík \cite{Frolik}, we see
by \cite[Theorem C]{Frolik} that for each convergent $\F$, its type
is produced by at most $2^{\omega}$ ultrafilters $\U$, and there
are $2^{2^{\omega}}$ ultrafilters, so convergent filters have $2^{2^{\omega}}$
different types. In particular, $\card(\lim_{\xi}^{-}(\infty))=2^{2^{\omega}}$
and $\mathsf{pp}(\xi)=2^{2^{\omega}}$. On the other hand, $\D(n)=\D(\{n\}^{\uparrow})$
converges to $\infty$, so that $\V_{\xi}(\infty)\leq\bigwedge_{n\in\omega}\D(n)$.
Conversely, suppose there is $S\in\bigwedge_{n\in\omega}\D(n)$ with
$S\notin\V_{\xi}(\infty)$, equivalently, $S^{c}\in\V_{\xi}(\infty)^{\#}$.
Because $S\in\bigwedge_{n\in\omega}\D(n)$ for each $n\in\omega$
there is $B_{n}\in\D(n)$ with $\bigcup_{n\in\omega}B_{n}\subset S$.
On the other hand, because $S^{c}\in\V_{\xi}(\infty)^{\#}$, there
is $\U\in\mathbb{U}\omega$ with $S^{c}\in\D(\U)$, that is, there
is $U\in\U$ with $S^{c}\in\bigwedge_{n\in U}\D(n)$, which is not
possible. Indeed, this means that for every $n\in U$, there is $A_{n}\in\D(n)$
with $\bigcup_{n\in U}A_{n}\subset S^{c}$, but then 
\[
\bigcup_{n\in\omega\setminus U}B_{n}\cup\bigcup_{n\in U}(B_{n}\cap A_{n})
\]
is a subset of $S$ that has non-empty intersection with $S^{c}$.
Hence $\V_{\xi}(\infty)=\bigwedge_{n\in\omega}\D(n)$ with $\infty\in\lim_{\xi}\D(n)$
for all $n$, so that $\infty\in\lim_{\De\xi}\V_{\xi}(\infty)$, that
is, $\De\xi=\S_{0}\xi$.
\end{example}

\begin{cor}
There is a set $X$ and a subset $S$ of $\mathbb{U}X$ which is not
$F_{\sigma}$ but has a dense $F_{\sigma}$-subset.
\end{cor}

There are convergences with countable pseudopaving number that are
not pretopologies (e.g., the sequential modification of the countable
sequential fan), hence not countably deep, and, as we have already
seen, countably deep convergences that are not pretopologies, hence
not of countable pseudopaving number. On the other hand, a pseudotopology
of countable pseudopaving number is a paratopology: 
\begin{thm}
\label{thm:SS1} If $\xi$ is a convergence with cover-Lindelöf singletons
(in particular if $\mathsf{pp}(\xi)\leq\omega$), then 
\[
\S\xi=\S_{1}\xi.
\]
In particular, a pseudotopology of countable pseudopaving number is
a paratopology.
\end{thm}

\begin{proof}
Though we could easily give a direct proof, we know by Corollary
\ref{cor:Lindelofptwise} that singletons are cover-Lindelöf if and
only if $\mathbb{U}_{\xi}(x)$ is Lindelöf for each $x$, hence $\delta$-closed.
In view of Corollary \ref{cor:SisS1}, we conclude that $\S\xi=\S_{1}\xi$.
Note that in view of Corollary \ref{cor:cppimpliesLindelofpt}, the
condition is satisfied in particular if $\mathsf{pp}(\xi)\leq\omega$.
\end{proof}
On the other hand, a paratopology may have non cover-Lindelöf singletons
(hence uncountable pseudopaving number), or may have countable paving
number without being a pretopology.
\begin{example}[A paratopology with a non cover-Lindelöf singleton]
 Let $X$ and $S$ be as in Corollary \ref{cor:deltaclosednotLindelof}.
Let $\xi$ be a prime convergence on $X$ with non-isolated point
$x_{0}$, defined by 
\[
x_{0}\in\lm_{\xi}\F\iff\beta\F\subset S\cup\{\{x_{0}\}^{\uparrow}\}.
\]
 Since $S$ is $\delta$-closed, so is $S\cup\{\{x_{0}\}^{\uparrow}\}$
and thus $\xi=\S_{1}\xi$ by Theorem \ref{thm:ADclosureinUX}. On
other hand, $S\cup\{\{x_{0}\}^{\uparrow}\}$ is not Lindelöf, so that
$\{x_{0}\}$ is not cover-Lindelöf by Corollary \ref{cor:Lindelofptwise}.
\end{example}

\begin{example}[A paratopology of countable paving number that is not a pretopology]
 Take $X$ and $S$ as in Corollary \ref{cor:hemicompactnotclosed}
and define on $X$ the prime pseudotopology with non-isolated point
$x_{0}$ for which 
\[
x_{0}\in\lm_{\xi}\F\iff\beta\F\subset S\cup\{\{x_{0}\}^{\uparrow}\}.
\]

In view of Proposition \ref{prop:paving}, this a pseudotopology with
countable paving number, hence a paratopology with countable paving
number by Theorem \ref{thm:SS1}. As $S\cup\{\{x_{0}\}^{\uparrow}\}$
is not closed, $\xi$ is not a pretopology.
\end{example}

\section{Hypotopologies}

Recall that a convergence $\xi$ is an \emph{hypotopology} if $\xi=\S_{\wedge1}\xi$.
The notion was introduced in\emph{ }\cite{D.comp} and is useful in
the study of the Lindelöf property. 

\begin{lem}
If $X$ is a countable set, then $\mathbb{F}_{\wedge1}X=\mathbb{F}_{0}X$.
\end{lem}

\begin{cor}
\label{cor:oncountablesets} For every convergence $\xi$ on a countable
set, $\S_{\wedge1}\xi=\S_{0}\xi$.
\end{cor}

It turns out that
\begin{thm}
\label{thm:hypocountablydeep} Every hypotopology is of countable
depth:
\[
\De\geq\De\triangle\S\geq\S_{\wedge1}.
\]
\end{thm}

\begin{proof}
If $x\in\lm_{\De\xi}\F$ then there is a sequence $\{\F_{n}\}_{n}$
of filters with $x\in\lm_{\xi}\F_{n}$ for every $n\in\mathbb{N}$
and $\F\geq\bigwedge_{n\in\mathbb{N}}\F_{n}$. Let $\H\in\mathbb{F}_{\wedge1}$
with $\H\#\F$. Then $\H\#\bigwedge_{n\in\mathbb{N}}\F_{n}$ so that,
there is $n\in\mathbb{N}$ with $\H\#\F_{n}$. Indeed, assume to the
contrary that for every $n\in\mathbb{N}$, there is $F_{n}\in\F_{n}$
with $F_{n}\notin\H^{\#}$, that is, $F_{n}^{c}\in\H$. Then $\bigcap F_{n}^{c}\in\H$
because $\H\in\mathbb{F}_{\wedge1}$ and the complement $\bigcup_{n\in\mathbb{N}}F_{n}$
of $\bigcap F_{n}^{c}$ belongs to $\bigwedge_{n\in\mathbb{N}}\F_{n}$
in contradiction to $\H\#\bigwedge_{n\in\mathbb{N}}\F_{n}$. Thus
$x\in\adh_{\xi}\H$ and we conclude that $x\in\lim_{\S_{\wedge1}\xi}\F$.
Thus $\De\geq\S_{\wedge1}$. 

Thus hypotopologies are countably deep pseudotopologies and thus the
finest countably deep pseudotopology coarser than a convergence is
finer than the finest hypotopology coarser than that convergence.
Hence $\De\triangle\S\geq\S_{\wedge1}$.
\end{proof}
\begin{rem}
Note that Example \ref{exa:cdeepnotpara} is a hypotopology. Indeed,
$x_{0}\in\lim_{\xi}\F$ if $\ker\F\subset\{x_{0}\}$ and $\F\neg\#(X)_{1}$,
that is, $\beta\F\cap\beta((X_{1}))=\emptyset$. Moreover, $\mathbb{F}_{\wedge1}^{\circ}=\{\H:\H\geq(X)_{1}\}$.
In particular, if $\H\in\mathbb{F}_{\wedge1}^{\circ}$ then $\adh_{\xi}\H=\emptyset$.
Let $x_{0}\in\lim_{\S_{\wedge1}\xi}\F$. Then the only countably deep
filters meshing with $\F$ are not free. If $x_{0}\notin\lim_{\xi}\F$
then $\F\#(X)_{1}$$|$ which is not the case, or $\ker\F\not\subset\{x_{0}\}$.
Let $x\neq x_{0}\in\ker\F$. Then $\{x\}^{\uparrow}\in\mathbb{F}_{\wedge1}$
and $\{x\}^{\uparrow}\#\F$ so $x_{0}\in\adh_{\xi}\{x\}^{\uparrow}=\lim_{\xi}\{x\}^{\uparrow}$
and thus $x=x_{0}$. So $x_{0}\in\lim_{\xi}\F$, that is, $\xi=\S_{\wedge1}\xi$.
\end{rem}

\begin{example}[A countably deep pseudotopology that is not a hypotopology (assuming
that there is a $P$-point in $\mathbb{U}^{\circ}\omega$)]
\label{exa:cdeeppstopnothypo} Suppose $\U_{0}$ is a $P$-point.
Let $\xi$ be the prime pseudotopology on $\omega\cup\{\infty\}$
defined by
\[
\infty\in\lm_{\xi}\F\iff\beta\F\subset\left(\mathbb{U}^{\circ}\omega\setminus\{\U_{0}\}\right)\cup\{\infty\}.
\]
It is not pretopological because $\mathbb{U}^{\circ}\omega\setminus\{\U_{0}\}$
is not closed, hence not hypotopological by Corollary \ref{cor:oncountablesets}.
On the other hand, $\xi$ is countably deep because $\mathbb{U}^{\circ}\omega\setminus\{\U_{0}\}$
contains the closure of its every $F_{\sigma}$-subsets.

In particular, we may have 
\[
\S\De\xi\geq\De\triangle\S\xi>\S_{\wedge1}\xi.
\]

\end{example}

\begin{cor}
\label{cor:S1ofhypo}
\[
\S_{1}\S_{\wedge1}=\S_{0}.
\]
\end{cor}

\begin{proof}
Of course $\S_{1}\S_{\wedge1}\geq\S_{0}$ because $\S_{1}\geq\S_{0}$
and $\S_{\wedge1}\geq\S_{0}$. In view of Theorem \ref{thm:hypocountablydeep},
\[
\S_{0}=\S_{1}\De\geq\S_{1}\S_{\wedge1},
\]
where the first equality follows from Proposition \ref{prop:S1D1}.
\end{proof}
Remarkably, while $\De\S_{1}\neq\S_{1}\De$ as observed in Example
\ref{exa:D1S1donotcommute}, we have in contrast 
\begin{thm}
\label{thm:hypocommute}
\[
\S_{\wedge1}\S_{1}=\S_{1}\S_{\wedge1}=\S_{0}.
\]
\end{thm}

\begin{proof}
In view of Corollary \ref{cor:S1ofhypo}, we only need to show that
$\S_{0}\geq\S_{\wedge1}\S_{1}$. To this end, we need to show that
$x\in\lim_{\S_{\wedge1}\S_{1}\xi}\V_{\xi}(x)$, equivalently that
\[
\beta(\V_{\xi}(x))=\cl_{\beta}(\mathbb{U}_{\xi}(x))\subset\mathbb{U}_{\S_{\wedge1}\S_{1}\xi}(x)=\cl_{(\mathbb{F}_{\wedge1})^{*}}\cl_{\delta}\mathbb{U}_{\xi}(x).
\]
 Suppose that $\U\notin\cl_{(\mathbb{F}_{\wedge1})^{*}}\cl_{\delta}\mathbb{U}_{\xi}(x)$,
equivalently, there is $\H\in\mathbb{F}_{\wedge1}$ with $\U\in\beta\H$
and $\beta\H\cap\cl_{\delta}\mathbb{U}_{\xi}(x)=\emptyset$, equivalently,
$\mathbb{U}_{\xi}(x)\notin(\N_{\delta}(\beta\H))^{\#}$. Since $\H\in\mathbb{F}_{\wedge1}$,
$\N_{\delta}(\beta\H)=\N_{\beta}(\beta\H)$ (See, e.g., \cite{mynard2007closure}).
Therefore, $\mathbb{U}_{\xi}(x)\notin(\N_{\beta}(\beta\H))^{\#}$
equivalently, $\beta\H\cap\cl_{\beta}\mathbb{U}_{\xi}(x)=\emptyset$.
In particular, $\U\notin\cl_{\beta}(\mathbb{U}_{\xi}(x))$.
\end{proof}
\bibliographystyle{plain}

\end{document}